%% file: 630Mei.tex
\documentclass{ecai}  

\usepackage{graphicx}
\usepackage{latexsym}

\usepackage{amsfonts}
\usepackage{booktabs}
\usepackage{multirow}
\usepackage{amsmath}
\usepackage{amssymb}
\usepackage{amsthm}
\usepackage{subcaption}
\usepackage{bm}
\usepackage{url}

\newtheorem{theorem}{Theorem}
\newtheorem{lemma}{Lemma}
\newtheorem{corollary}{Corollary}

\newtheorem{example}{Example}[theorem]

\newcommand{\dd}{\mathop{}\!{d}} 
\newcommand{\p}{\mathop{}\!{\partial}}
\newcommand{\T}{^\mathrm{T}} 
\newcommand{\xstar}{\bm{x}_{1:k}}
\newcommand{\x}{\bm{x}_{k+1}}
\newcommand{\y}{\bm{y}_{k+1}}
\newcommand{\eqdef}{\stackrel{\mathrm{def}}{=}}

\begin{document}

\begin{frontmatter}

\title{A Bayesian Optimization Framework for Finding Local Optima in Expensive Multimodal Functions}


\author[A]{\fnms{Yongsheng}~\snm{Mei}\thanks{Corresponding Author. Email: ysmei@gwu.edu.}}
\author[A]{\fnms{Tian}~\snm{Lan}}
\author[B]{\fnms{Mahdi}~\snm{Imani}}
\author[A]{\fnms{Suresh}~\snm{Subramaniam}}

\address[A]{The George Washington University}
\address[B]{Northeastern University}

\begin{abstract}
Bayesian optimization (BO) is a popular global optimization scheme for sample-efficient optimization in domains with expensive function evaluations. The existing BO techniques are capable of finding a single global optimum solution. However, finding a set of global and local optimum solutions is crucial in a wide range of real-world problems, as implementing some of the optimal solutions might not be feasible due to various practical restrictions (e.g., resource limitation, physical constraints, etc.). In such domains, if multiple solutions are known, the implementation can be quickly switched to another solution, and the best possible system performance can still be obtained. This paper develops a multimodal BO framework to effectively find a set of local/global solutions for expensive-to-evaluate multimodal objective functions. We consider the standard BO setting with Gaussian process regression representing the objective function. We analytically derive the joint distribution of the objective function and its first-order derivatives. This joint distribution is used in the body of the BO acquisition functions to search for local optima during the optimization process. We introduce variants of the well-known BO acquisition functions to the multimodal setting and demonstrate the performance of the proposed framework in locating a set of local optimum solutions using multiple optimization problems.
\end{abstract}

\end{frontmatter}

\input{introduction.tex}

\input{background.tex}

\input{methodology.tex}

\input{experiment.tex}

\ack This research is based on work supported by the Army Research Office (ARO) under Grant W911NF2110299.

\newpage
\bibliography{reference}

\input{appendix.tex}

\end{document}

%% file: introduction.tex
\section{Introduction}
\label{sec:introduction}

Bayesian optimization (BO) is a popular global optimization scheme for sample-efficient optimization in domains with expensive function evaluations~\cite{shahriari2015taking,wu2017bayesian}. The BO iteratively builds a statistical model of the objective function according to all the past evaluations and sequentially selects the next evaluation by maximizing an acquisition function. BO has shown tremendous success in a wide range of domains with no analytical formulation of objective functions, including simulation optimizations~\cite{acerbi2017practical}, device tuning/calibration~\cite{dalibard2017boat,vargas2020bayesian}, material/drug design~\cite{zhang2020bayesian,griffiths2020constrained}, and many more. 

Despite several variants of BO in recent years~\cite{balandat2020botorch,eriksson2019scalable,wu2019hyperparameter}, the focus of all methods has been on finding a single global optimum solution. However, many real-world problems can be considered multimodal optimization~\cite{das2011real}, where it is desired to find all or most global and local optimum solutions of multimodal functions. The rationale is that implementing some of the optimal solutions might not be feasible due to various practical restrictions (e.g., resource limitation, physical constraints, etc.); in such a scenario, if multiple solutions are known, the implementation can be quickly switched to another solution and the best possible system performance can still be obtained.

Besides BO techniques, several deterministic and stochastic techniques have been developed for multi-model optimization. These include the gradient descent method, the quasi-Newton method~\cite{byrd2016stochastic,lewis2013nonsmooth}, and the Nelder-Mead's simplex method~\cite{butt2017globalized}, which require the analytical form of the multimodal function and tend to be trapped into a local optimum. Evolutionary optimization is a class of techniques applicable to domains with no available analytical representations. Examples include variants of genetic algorithms~\cite{li2015truss,liang2011genetic}, clonal selection algorithms~\cite{de2000clonal}, and artificial immune networks~\cite{de2002artificial}. However, evolutionary techniques' reliance on heuristics and excessive function evaluations prevent their reliable applications in domains with expensive-to-evaluate functions.

This paper develops a multimodal BO framework to effectively find a set of local/global solutions for multimodal objective functions. We consider the standard Gaussian process (GP) regression as the surrogate model for representing the objective function. Since characterizing local/global optima requires accessing both the values and first-order conditions of the objective function, we analytically derive the joint distribution of the objective function and its first-order gradients using the kernel function properties and derivatives, illustrated by Fig.~\ref{fig:framework}. This joint distribution is used in the body of the BO acquisition functions to search for local optima during the optimization process. We introduce the variants of the well-known BO acquisition functions to the multimodal setting, including joint expected improvement and probability of improvement. The performance of the proposed framework in locating a set of local optimum solutions is demonstrated using multiple optimization problems, including evaluations using synthetic function, well-known multimodal benchmarks such as Griewank function and Shubert function, as well as through hyperparameter tuning problems for image classification on CIFAR-10 dataset~\cite{krizhevsky2009learning}. Our proposed solution can effectively capture multiple local/global optima in these evaluations. 

The main contributions of our work are as follows:
\begin{itemize}
	\item We develop a new BO framework for finding local optima by analytically deriving the joint distribution of the objective function and its first-order gradients according to the kernel function of Gaussian process regression. 
	\item We introduce new acquisition functions, such as joint expected improvement and probability of improvement, to search local optima during the optimization process of our framework.
	\item Experimental results on multimodal functions and real-world image classification problem demonstrates the effectiveness of our framework in capturing multiple local and global optima.
\end{itemize}

%% file: background.tex
\section{Related Work}

\subsection{Bayesian Optimization and Acquisition Functions}

Among many optimization frameworks~\cite{wang2021network,rafiee2021intermittent,wu2023serverless}, Bayesian optimization has emerged as a popular method for the sample-efficient optimization of expensive objective functions~\cite{snoek2012practical,frazier2018tutorial}. It is capable of optimizing objective functions with no available closed-form expression, where only point-based (and possibly noisy) evaluation of the function is possible. The sampling-efficiency of BO schemes has made them suitable for domains with non-convex and costly function evaluations. Fundamentally, BO is a sequential model-based approach to solving the problem of finding a global optimum of an unknown objective function $f(\cdot) $: $\boldsymbol{x}^* = \arg\max_{x \in \mathcal{X}} f(\boldsymbol{x})$, where $ {\mathcal X} \subset {\mathbb R}^d $ is a compact set. It performs a sequential search, and at each iteration $k$, selects a new location $x_{k+1}$ to evaluate $f$ and observe its value, until making a final recommendation of the best estimate of the optimum $\boldsymbol{x}^*$.

The sequential selection is achieved through the acquisition function $a:\mathcal{X} \rightarrow \mathbb{R}$, defined over the posterior of GP model, where BO selects a sample in the search space with the highest acquisition value. Upon evaluating the objective function at the selected input, the surrogate GP model gets updated. Various BO policies have been developed depending on the acquisition function choices, including expected improvement (EI)~\cite{movckus1975bayesian}, knowledge gradient (KG)~\cite{frazier2009knowledge}, probability of improvement (PI)~\cite{kushner1964new}, upper-confidence bounds (UCB)~\cite{lai1985asymptotically}, and entropy search (ES)~\cite{hernandez2014predictive}. Recent studies have also considered accounting for future improvements in solution quality~\cite{zhang2021two,jiang2020efficient,wu2019practical} and BO methods for constrained problems~\cite{ariafar2019admmbo,perrone2019constrained}. However, existing work often focuses on finding the global optimum rather than a set of local/global optima, which are also necessary for multimodal objective functions and will be explored in this paper.

\subsection{Local Maxima in Gaussian Random Fields}

A separate line of statistical applications concentrates on the tail distribution of the heights of local maxima in non-stationary Gaussian random fields, such as in peak detection problems. Such a tail distribution is defined as the probability that the height of the local maximum surpasses a given threshold at the point $ \bm{x} $, conditioned on the case that the point is a local maximum of the Gaussian process model, given by the following equation: $ \Pr(f(\bm{x}>\xi)|f(\bm{x}) \text{ is one local maximum}) $, where $ \xi $ is the threshold. Existing work~\cite{cheng2015distribution} has probed into this problem in the Gaussian process model. The general formulae are derived in~\cite{cheng2015distribution} for non-stationary Gaussian fields and a subset of Euclidean space or Riemannian manifold of arbitrary dimension~\cite{lee2018introduction}, which depends on local properties of the Gaussian process model rather than the global supremum of the field. Although the goal is to characterize certain local properties rather than finding a set of local/global optima, the contributions of the above work provide the key intuitions to our new BO framework design.

\subsection{First-Order Derivative in Bayesian Optimization}

The first-order derivative information has been exploited in existing works of BO, such as~\cite{lizotte2008practical,vzilinskas2006practical}. Besides, \cite{maclaurin2015gradient} make gradient information available for hyperparameter tuning. Additionally, adjoint methods provide gradients cheaply in the optimization of engineering systems~\cite{jameson1999re,plessix2006review}. The gradients provide useful information about the objective function and can help the BO during the optimization process. For instance, \cite{wu2017bayesian} develop a derivative-enabled knowledge-gradient algorithm by incorporating derivative information into GP for BO, given by $ (f(\x),\nabla f(\x))\T|f(\xstar),\nabla f(\xstar) \sim \mathcal{N}\bigl((f(\xstar),\nabla f(\xstar))\T, \mathrm{diag}(\sigma^2)\bigr) $, where $ \sigma^2 $ is the variance. These methods assume the gradients of the objective function can be queried along with the objective function during the optimization process. GIBO\cite{muller2021local} alternates between minimizing the variance of the estimate of the gradient and moving in the direction of the expected gradient. Later proposed MPD \cite{nguyen2022local} extended and refined it by finding the maximum look-ahead gradient descent direction. However, we propose the BO framework to find local optima by computing and updating the joint distribution of the prediction with its first-order derivative regarding kernel functions. Besides, existing methods with gradients concentrate on using gradients as additional information to improve the traditional BO model targeting global optimum, while our algorithm aims to find as many local optima as possible. Despite the similarity, we do not directly access the objective's gradient.

\section{Background}

{\em Gaussian model}: Gaussian process (GP) model is the most commonly used model for standard BO and also adopted by our framework, providing the posterior distribution of the objective function according to all the past evaluations. In this part, we introduce several preliminaries of the GP model. Considering the objective function $ f(\cdot) $ and the GP model with $ k + 1 $ input samples $ \bm{x} $ of dimension $ n $, the prior of the model is:
\begin{equation*}
	f(\bm{x}_{1:k+1}) \sim \mathcal{N}(\mu_{\bm{x}_{1:k+1}},\Sigma_{\bm{x}_{1:k+1},\bm{x}_{1:k+1}}),
\end{equation*}
where we use $ \mu_{\bm{x}_{1:k+1}} $ to denote the mean of the prior and $ \Sigma_{\bm{x}_{1:k+1},\bm{x}_{1:k+1}} $ to represent the initial covariance of $ k+1 $ input samples. If we know the first $ k $ samples' values as observations $ f(\xstar) $, based on the prior, the posterior of the GP model representing the objective function at the next sampling point $ \x $ can be obtained as:
\begin{equation}
	p_f \eqdef f(\x)|f(\xstar) \sim \mathcal{N}(\mu_{\x},\Sigma_{\x,\x}),
	\label{eq:f_x}
\end{equation}
where $ \mu_{\x} $ and $ \Sigma_{\x,\x} $ are the mean and variance respectively at this step. For simplicity purpose, we use $ p_f $ as the short notation of the posterior $ f(\x)|f(\xstar) $.

%% file: methodology.tex
\section{Finding Local Optima via Bayesian Optimization}
\label{sec:methodology}

\begin{figure}[t]
	\centering
	\begin{subfigure}[t]{0.40\textwidth}
		\centering
		\includegraphics[width=\textwidth]{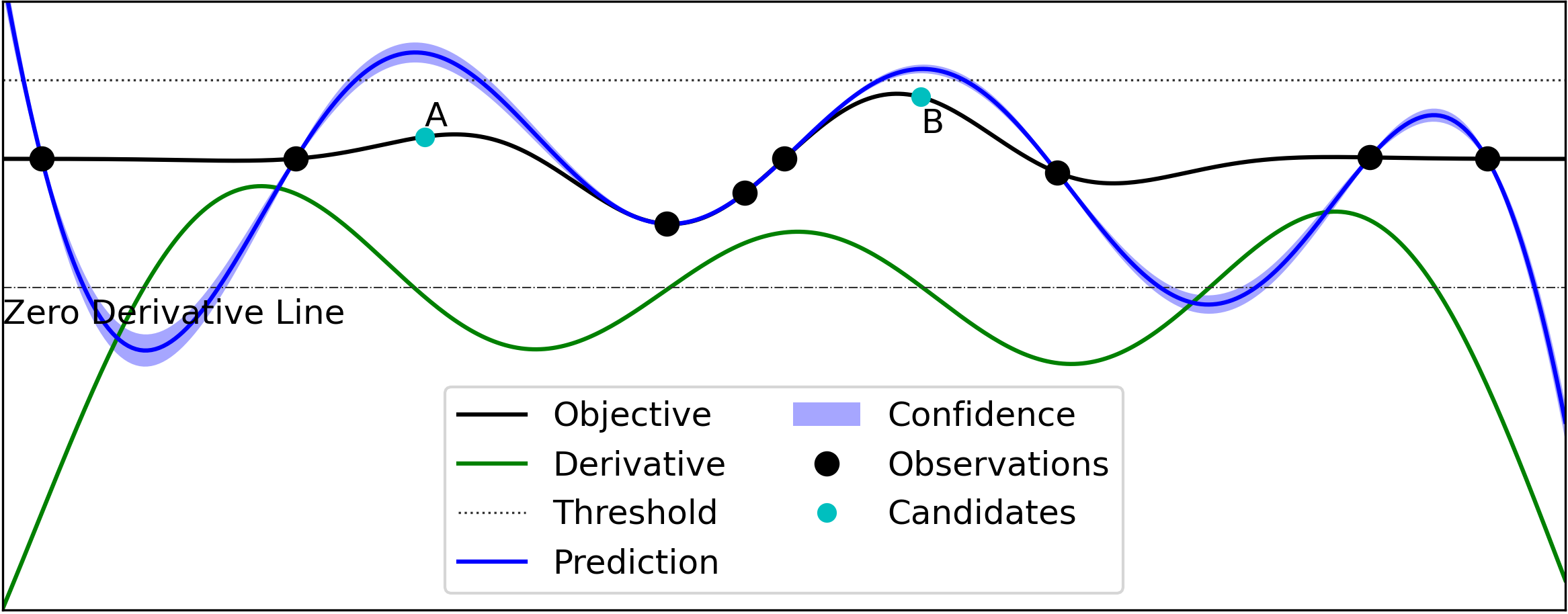}
		\caption{\small Iteration $ k$}
		\label{fig:framework_k}
	\end{subfigure}
	\begin{subfigure}[t]{0.40\textwidth}
		\centering
		\includegraphics[width=\textwidth]{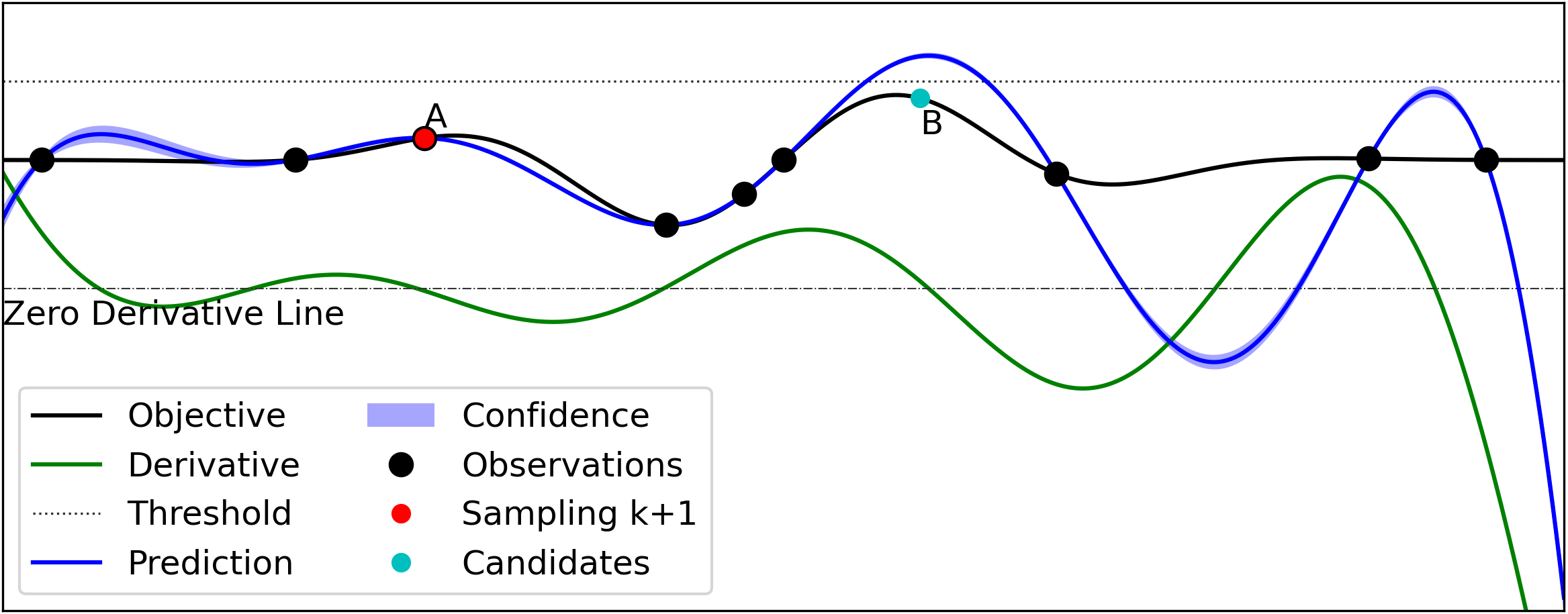}
		\caption{\small Iteration $ k+1 $}
		\label{fig:framework_k+1}
	\end{subfigure}
	\caption{Illustration of our proposed BO method for finding a set of local/global optima by considering both GP model prediction and its first-order derivatives.}
	\label{fig:framework}
\end{figure}

We aim to develop a multimodal BO framework capable of efficiently computing optimal solutions. Our method will achieve optimization regarding a joint distribution containing the Gaussian posterior and its first-order derivative: the Gaussian posterior reflects the surrogate model as we observe the objectives in new places, and the first-order derivative of the posterior informs us about the latent location of local/global optima. Therefore, to this end, we compute the first-order derivative of the objective function to get the gradient and apply it to our new BO framework. Since the derivative operation is linear, we can combine the Gaussian posterior of the objective and its first-order derivative into a joint distribution and then compute the joint mean/variance for the acquisition functions. In this framework, we design new acquisition functions to fulfill the needs of finding local optima. Inspired by existing work~\cite{cheng2015distribution}, we first introduce a particular threshold for the objective posterior to find candidates above/below the threshold value (depending on finding maxima/minima). Besides, as the optimum point's gradient is usually zero, we limit the posterior of the calculated derivative to a small interval around zero to further approach the optimum.

{\em An illustrative example.} Fig.~\ref{fig:framework} shows an illustrative example of our proposed method. In iteration $k$, we update the Gaussian posterior of the objective and its first-order derivative. Existing BO methods for finding a global maximum would select $\bm{x}_{k+1}$ (and subsequent locations) to maximize the expected improvement or the probability of improvement, failing to recognize other local/global maxima in this multimodal function. For instance, probability of improvement can be described by the equation: $ \Pr(f(\bm{x})>\xi) $, where $ \xi $ is the given threshold. 

To find other local/global optima, we derive the prediction's first-order derivative and let it satisfies: $ \Pr(f'(\bm{x})<\epsilon) $, where $ \epsilon $ denotes a small constant and hence indicates the place where the first-order condition approximates to zero. 

On the other hand, if we only consider the first-order condition in choosing $\bm{x}_{k+1}$, the sequential search could easily get trapped in stationary/saddle points without improving the objective value toward local maxima. Therefore, our method leverages the joint distribution of objective function and its first-order derivative to find a set of local/global maxima by jointly computing and updating the probability: $ \Pr(f(\bm{x})>\xi, f'(\bm{x})<\epsilon) $.

As shown in Fig.~\ref{fig:framework_k}, we have two latent local/global maxima, A and B, denoted by the cyan dots. Given the threshold $ \xi $ and first-order condition, the proposed algorithm will select local maximum A as the next sampling point rather than global maximum B at iteration $ k $. Although our method can find the local maximum, it is not oblivious to other possible solutions, given by Fig.~\ref{fig:framework_k+1}, which suggests B as the latent sampling point for the next optimum solution. This illustrative example demonstrates our method's capability of finding all latent optimum solutions of multimodal functions by considering both GP model prediction and the first-order condition. Furthermore, analyzing such joint distribution would also be useful for deriving other special acquisition functions such as the expected improvement.

\subsection{Multimodal Bayesian Optimization}

We characterize a local/global optimum solution by leveraging the estimates of both the objective function and the first-order derivative of the function in the BO framework. We begin by considering the first-order derivative of the objective function $ f(\cdot) $. Due to the linear property of the differentiation operation, we have the derivative of $ f(\cdot) $ also subject to the GP model~\cite{rasmussen2006gaussian}. Then, given input sample $ \y $ that can be a different point from $ \x $, we derive another posterior:
\begin{equation}
	p_{f'} \eqdef f'(\y)|f(\xstar) \sim \mathcal{N}(\mu_{\y},\Sigma_{\y,\y}),
	\label{eq:fprime_y}
\end{equation}
where, for simplicity, we let $ p_f' $ be short for the posterior $ f'(\y)|f(\xstar) $.

Since both $ \x $ and $ \y $ are related to the objective function $ f(\cdot) $, posteriors $ f(\x)|f(\xstar) $ and $ f'(\y)|f(\xstar) $ will not be independent and identically distributed (i.i.d.). Besides, based on~\eqref{eq:f_x} and~\eqref{eq:fprime_y}, and as both posteriors comply with the Gaussian distribution, we can further acquire a joint Gaussian distribution as follows:
\begin{equation}
	p_{f,f'} \sim \mathcal{N}\left(
	\begin{bmatrix}
		\mu_{\x} \\
		\mu_{\y}
	\end{bmatrix},
	\begin{bmatrix}
		\Sigma_{\x,\x} & \Sigma_{\x,\y} \\
		\Sigma_{\y,\x} & \Sigma_{\y,\y}
	\end{bmatrix}\right).
	\label{eq:joint_posterior}
\end{equation}

Before solving the mean and variance in joint distribution~\eqref{eq:joint_posterior}, we introduce the sufficient condition leveraging dominated convergence theorem in a lemma for interchanging derivative with expectation.

\begin{lemma}[Interchangeable condition]
	Let $ Y \in \mathcal{Y} $ be a random variable. $ g:\mathbb{R}\times\mathcal{Y} \rightarrow \mathbb{R} $ is a function such that $ g(t,Y) $ is integrable for all $ t $ and $ g $ is continuously differentiable with respect to $ t $. Assume that there is a random variable $ Z $ such that $ |\frac{\p}{\p t}g(t,Y)| \le Z $ almost surely for all $ t $ and $ \mathbb{E}(Z)<\infty $. Then, we have:
	\begin{equation*}
		\frac{\p}{\p  t}\mathbb{E}\left[g(t,Y)\right]=\mathbb{E}\left[\frac{\p}{\p t}g(t,Y)\right].
	\end{equation*}
	\label{lem:dct}
\end{lemma}

\begin{proof}
See Appendix~\ref{subsec:proof_lem}.
\end{proof}

Next, we derive the mean and variance of the target joint distribution~\eqref{eq:joint_posterior}. Based on standard mean and variance results of Gaussian posterior model, we can obtain our mean of designed algorithm given by the following steps:
\begin{equation}
	\begin{aligned}
		\mu_{k+1} &= \mathbb{E}
		\begin{bmatrix}
			f(\x)|f(\xstar) \\
			f'(\y)|f(\xstar)
		\end{bmatrix} \\
		&\stackrel{(a)}{=}\frac{\dd}{\dd \y} \mathbb{E}
		\begin{bmatrix}
			f(\x)|f(\xstar) \\
			f(\y)|f(\xstar)
		\end{bmatrix} \\
		&\;= 
		\resizebox{.84\linewidth}{!}{$
		\begin{bmatrix}
			k(\x,\xstar) \\
			\frac{\dd}{\dd{\y}} k(\y,\xstar)
		\end{bmatrix} \mathbf{K}^{-1}_{\xstar,\xstar} (f(\xstar) - \mu_0(\xstar)),
		$}
	\end{aligned}
	\label{eq:mean_1}
\end{equation}
where (a) uses the interchangeable condition in Lemma~\ref{lem:dct}, and $ \mu_0 $ denotes the initial mean of the first $ k $ samples. 

Furthermore, as the joint posterior is constructed by objective prediction and its first-order condition, the first covariance matrix $ \mathbf{K}_0 $ is given by:
\begin{equation*}
	\mathbf{K}_0 =
	\begin{bmatrix}
		k(\x,\x) & \frac{\dd}{\dd{\y}}k(\x,\y) \\
		\frac{\dd}{\dd{\y}}k(\y,\x) & \frac{\dd^2}{\dd{\y}\dd{\y}}k(\y,\y)
	\end{bmatrix}
\end{equation*}

Therefore, according to the variance of GP model, the variance of~\eqref{eq:joint_posterior} is:
\begin{equation}
	\begin{aligned}
		&\sigma^2_{k+1} \\
		&= 
		\resizebox{0.95\linewidth}{!}{$
		\mathbf{K}_0- 	
		\begin{bmatrix}
			k(\x,\xstar) \\
			\frac{\dd}{\dd{\y}} k(\y,\xstar)
		\end{bmatrix} \mathbf{K}^{-1}_{\xstar,\xstar} 	
		\begin{bmatrix}
			k(\x,\xstar) \\
			\frac{\dd}{\dd{\y}} k(\y,\xstar)
		\end{bmatrix}\T. $}
	\end{aligned}
\end{equation}

Therefore, we conclude the mean and variance of joint distribution~\eqref{eq:joint_posterior} in a theorem by leveraging Lemma~\ref{lem:dct} as follows:

\begin{theorem}[Bayesian optimization for local optima]
	Under the interchangeable condition, the mean of the joint Gaussian distribution posterior in~\eqref{eq:joint_posterior} is:
	\begin{equation}
		\begin{aligned}
			\mu_{k+1} &\eqdef
			\begin{bmatrix}
				\mu_{\x} \\
				\mu_{\y}
			\end{bmatrix} \\ &\;
			= \mathbf{A} \mathbf{K}^{-1}_{\xstar,\xstar} (f(\xstar) - \mu_0(\xstar)),
		\end{aligned}
		\label{eq:mean_2}
	\end{equation}
	where $ \mu_0 $ denotes the initial mean of the first $ k $ samples, and the auxiliary matrix $ \mathbf{A} $ is provided as follows:
	\begin{equation*}
		\mathbf{A} =
		\begin{bmatrix}
			k(\x,\xstar) \\
			\frac{\dd}{\dd{\y}} k(\y,\xstar)
		\end{bmatrix}.
	\end{equation*}
	
	Additionally, the solution of variance in the posterior is given by:
	\begin{equation}
		\begin{aligned}
			\sigma^2_{k+1} &\eqdef 
			\begin{bmatrix}
				\Sigma_{\x,\x} & \Sigma_{\x,\y} \\
				\Sigma_{\y,\x} & \Sigma_{\y,\y}
			\end{bmatrix} \\ &
			\;
			= \mathbf{K}_0 - \mathbf{A} \mathbf{K}^{-1}_{\xstar,\xstar} \mathbf{A}\T.
		\end{aligned}
		\label{eq:var}
	\end{equation}
	\label{theo:local}
\end{theorem}

In Theorem~\ref{theo:local}, $ k(\cdot) $ in \eqref{eq:mean_2} denotes the kernel function, such as the squared-exponential and Matern kernels, which define the influence of a solution on the performance and confidence estimations of untested nearby solutions. Apart from that, in variance~\eqref{eq:var}, $ \mathbf{K} $ represents the covariance matrix, where $ [\mathbf{K}_{\xstar,\xstar}]_{ij} = k(\boldsymbol{x}_i,\boldsymbol{x}_j) $

Furthermore, we provide two examples using the squared-exponential kernel and polynomial kernel, respectively, to illustrate Theorem~\ref{theo:local} as below.

\begin{example}[Square-exponential kernel]
	As a popular kernel function widely used in many existing works, the squared-exponential kernel can strongly connect spatially adjacent sampling points as those points are more similar. The kernel is defined by the equation below:
	\begin{equation}
		k(\boldsymbol{x}_i,\boldsymbol{x}_j)=\alpha \exp(-\frac{\Vert \boldsymbol{x}_i-\boldsymbol{x}_j \Vert_2^2}{2l^2}),
		\label{eq:rbf}
	\end{equation}
	where $ \alpha $ and $ l $ are scale factor and length scale parameter introduced by the kernel, respectively.
	
	For simplification, we assume $ \y=\x $. By leveraging the given square-exponential kernel~\eqref{eq:rbf} and computing its gradient, the mean in Theorem~\ref{theo:local} becomes:
	\begin{equation}
		\mu_{k+1} = \Bar{\mathbf{A}} \mathbf{K}^{-1}_{\xstar,\xstar} f(\xstar),
		\label{eq:rbf_mean}
	\end{equation}
	where $ f(\xstar) $ are known initial sample values and $ \Bar{\mathbf{A}} $ is given by:
	\begin{equation*}
		\Bar{\mathbf{A}} = \alpha
		\begin{bmatrix}
			\exp(-\frac{\Vert \x-\xstar \Vert_2^2}{2l^2}) \\
			-\frac{\x-\xstar}{l^2} \exp(-\frac{\Vert \x-\xstar \Vert_2^2}{2l^2})
		\end{bmatrix}.
	\end{equation*}
	
	Likewise, we obtain the variance as:
	\begin{equation}
		\begin{aligned}
			\sigma^2_{k+1} = \alpha
			\begin{bmatrix}
				1 & \mathbf{0} \\
				\mathbf{0} & \frac{1}{l^2} \mathbf{I}_n
			\end{bmatrix} -
			\Bar{\mathbf{A}} \mathbf{K}^{-1}_{\xstar,\xstar} \Bar{\mathbf{A}}\T,
		\end{aligned}
		\label{eq:rbf_var}
	\end{equation}
	where $ \mathbf{I}_n $ represents the identity matrix of dimension $ n $.
\end{example}

\begin{example}[Polynomial kernel]
	In this example, we consider another frequently-used kernel, the Polynomial kernel, which represents the similarity of sampling points in the space over polynomials, allowing the learning of non-linear models. The kernel is defined as:
	\begin{equation}
		k(\boldsymbol{x}_i,\boldsymbol{x}_j)=\bar{\alpha} (\boldsymbol{x}_i \cdot \boldsymbol{x}_j - c)^\delta,
		\label{eq:poly}
	\end{equation}
	where $ c \ge 0 $ is a free parameter trading off the influence of higher-order versus lower-order terms in the polynomial, $ \bar{\alpha} $ is the scale factor, and $ \delta $ denotes the polynomial's dimension.
	
	For simplification, let $ \y=\x $. In the most typical situation, the kernel is a homogeneous polynomial kernel where $ c=0 $, and the dimension $ \delta=2 $. Then, we compute its gradient by leveraging the given polynomial kernel~\eqref{eq:poly}, and the mean in Theorem~\ref{theo:local} becomes:
	\begin{equation}
		\mu_{k+1} = \Bar{\Bar{\mathbf{A}}} \mathbf{K}^{-1}_{\xstar,\xstar} f(\xstar),
		\label{eq:poly_mean}
	\end{equation}
	where $ f(\xstar) $ are known initial sample values and $ \Bar{\Bar{\mathbf{A}}} $ is given by:
	\begin{equation*}
		\Bar{\Bar{\mathbf{A}}} = \bar{\alpha}
		\begin{bmatrix}
			(\x \cdot \xstar)^2 \\
			2(\x \cdot \xstar) \cdot \xstar
		\end{bmatrix}.
	\end{equation*}
	
	Likewise, we obtain the variance as:
	\begin{equation}
		\begin{aligned}
			&\sigma^2_{k+1} \\
			&= \bar{\alpha}
			\begin{bmatrix}
				\Vert \x \Vert^4 &  2\Vert \x \Vert^2 \cdot \x \\
				2\Vert \x \Vert^2 \cdot \x & 2(\x \otimes \x + \Vert \x \Vert^2 \mathbf{I}_n)
			\end{bmatrix} \\ 
			&\quad - \Bar{\Bar{\mathbf{A}}} \mathbf{K}^{-1}_{\xstar,\xstar} \Bar{\Bar{\mathbf{A}}}\T,
		\end{aligned}
		\label{eq:poly_var}
	\end{equation}
	where $ \otimes $ denotes the outer product and $ \mathbf{I}_n $ represents the identity matrix of dimension $ n $.
\end{example}

\subsection{Determining Local Optima}

We manage to access the posterior of any $ \x $ via the designed GP model. The next step is determining the next test point to evaluate, which can be obtained through the acquisition function (AF). Due to the joint distribution in~\eqref{eq:joint_posterior}, to guarantee the discovery of a local optimum solution, we need 1) the value of the objective function at the test point to be larger than elsewhere nearby, and 2) the first-order derivative of the objective function at the same point to be close to zero. This step is another optimization problem regarding $ p_{f,f'} $, but does not require evaluating the objective function $ f(\cdot) $. 

To start with, we use conditional distribution expansion to change the form of the joint probability distribution. The new joint posterior becomes:
\begin{equation}
	p_{f,f'} = (p_f|p_{f'}) \cdot p_{f'},
	\label{eq:expand}
\end{equation}
where $ p_f|p_{f'} $ obeys the Gaussian distribution with the mean of:
\begin{equation}
	\bar{\mu}_{k+1} = \mu_{\x}+\Sigma_{\x,\y}\Sigma_{\y,\y}^{-1}(f'(\y)-\mu_{\y}),
	\label{eq:af_mean}
\end{equation}
and variance of:
\begin{equation}
	\bar{\sigma}^2_{k+1} = \Sigma_{\x,\x}-\Sigma_{\x,\y}\Sigma_{\y,\y}^{-1}\Sigma_{\y,\x}.
	\label{eq:af_var}
\end{equation}

To satisfy the two aforementioned requirements, we let $ \x $ and $ \y $ be at the same position, where we have: 
\begin{equation*}
	p_{f'} = f'(\y)|f(\xstar) = f'(\x)|f(\xstar).
\end{equation*}

According to~\eqref{eq:expand}, we further expand the joint distribution in~\eqref{eq:joint_posterior} to fulfill the needs of the AF design. After that, referring to the original AF definition, we newly define several AFs for determining the local optimum solution given in the following corollaries as examples.

Firstly, we provide the joint probability of improvement, which evaluates the objective function at the point most likely to improve upon the value of the current observations. Under such criterion, the point with the highest probability of improvement will be selected.

\begin{corollary}[Joint probability of improvement]
	In our case, when $ \x=\y $, the AF objective is:
	\begin{equation}
		\begin{aligned}
			&a_{PI}(\x) \\
			&\stackrel{(b)}{\approx} \int_{\xi}^{\infty} (p_f|p_{f'}) \dd{\xi} \cdot \int_{-\epsilon}^{\epsilon} p_{f'} \dd{\epsilon} \\
			&= \resizebox{.95\hsize}{!}{$ Q\left(\frac{\xi-\bar{\mu}_{k+1}}{\bar{\sigma}^2_{k+1}}\right)\left[Q\left(\frac{-\epsilon-\mu_{\y}}{\Sigma_{\y,\y}}\right)-Q\left(\frac{\epsilon-\mu_{\y}}{\Sigma_{\y,\y}}\right)\right], $}
		\end{aligned}
		\label{eq:pi}
	\end{equation}
	where (b) uses expansion~\eqref{eq:expand} and an approximation in calculating integral, $ \xi $ is the probability of improvement threshold, $ \epsilon $ is a small constant to restrict the first-order derivative, and $ Q(\cdot) $ denotes the Q-function with $ \bar{\mu}_{k+1} $ and $ \bar{\sigma}_{k+1} $ given in~\eqref{eq:af_mean} and~\eqref{eq:af_var}, respectively.
	\label{coro:pi}
\end{corollary}

Note that we split the original integral function and let $ p_f|p_{f'} \ge \xi $ and $ |p_{f'}| \le \epsilon $ to suit our purpose in determining local optimum solution. In (a) of~\eqref{eq:pi}, when calculating acquisition, due to small $ \epsilon $, the standard double integral can be approximated by the product of two integrals. As $ p_f|p_{f'} $ and $ p_{f'} $ follow Gaussian distribution, using Q-function, which is the tail distribution function of the standard Gaussian distribution, simplifies the AF in the expression.

Next, we define a joint expected improvement that evaluates the objective function at the point that improves upon the value of the current observations in terms of expectation. Under this criterion, the point with the greatest expected improvement will be selected.

\begin{corollary}[Joint expected improvement]
	In our setting, when $ \x=\y $, the AF objective becomes:
	\begin{equation}
		\begin{aligned}
			&a_{EI}(\x) \\
			&\stackrel{(c)}{\approx} \int_{\xi}^{\infty}(f(\x)-\xi) (p_f|p_{f'}) \dd{\xi} \cdot \int_{-\epsilon}^{\epsilon} p_{f'} \dd{\epsilon} \\
			&= \int_{\xi}^{\infty}(f(\x)-\xi) (p_f|p_{f'}) \dd{\xi} \\
			&\quad \cdot \left[Q\left(\frac{-\epsilon-\mu_{\y}}{\Sigma_{\y,\y}}\right)-Q\left(\frac{\epsilon-\mu_{\y}}{\Sigma_{\y,\y}}\right)\right],
		\end{aligned}
	\end{equation}
	where (c) adopts expansion~\eqref{eq:expand} and an approximation in calculating integral. Other parameters are defined the same as those in Corollary~\ref{coro:pi}.
	\label{coro:ei}
\end{corollary}

Two AFs introduced in the corollaries will be used with our new BO framework for determining the local optima.

%% file: experiment.tex
\section{Experiment}
\label{sec:experiment}

In this section, we perform thorough experiments on the synthetic multimodal functions and real-world implementation to test our algorithm's effectiveness in finding local optima. Additionally, we report the ablations and scalability experiments of the designed algorithm on multimodal functions. The code has been made available at: \url{https://github.com/ysmei97/local_bo}.

\begin{figure*}[!htb]
	\begin{minipage}{0.36\textwidth}
		\centering
		\begin{subfigure}[t]{0.49\textwidth}
			\centering
			\includegraphics[width=\textwidth]{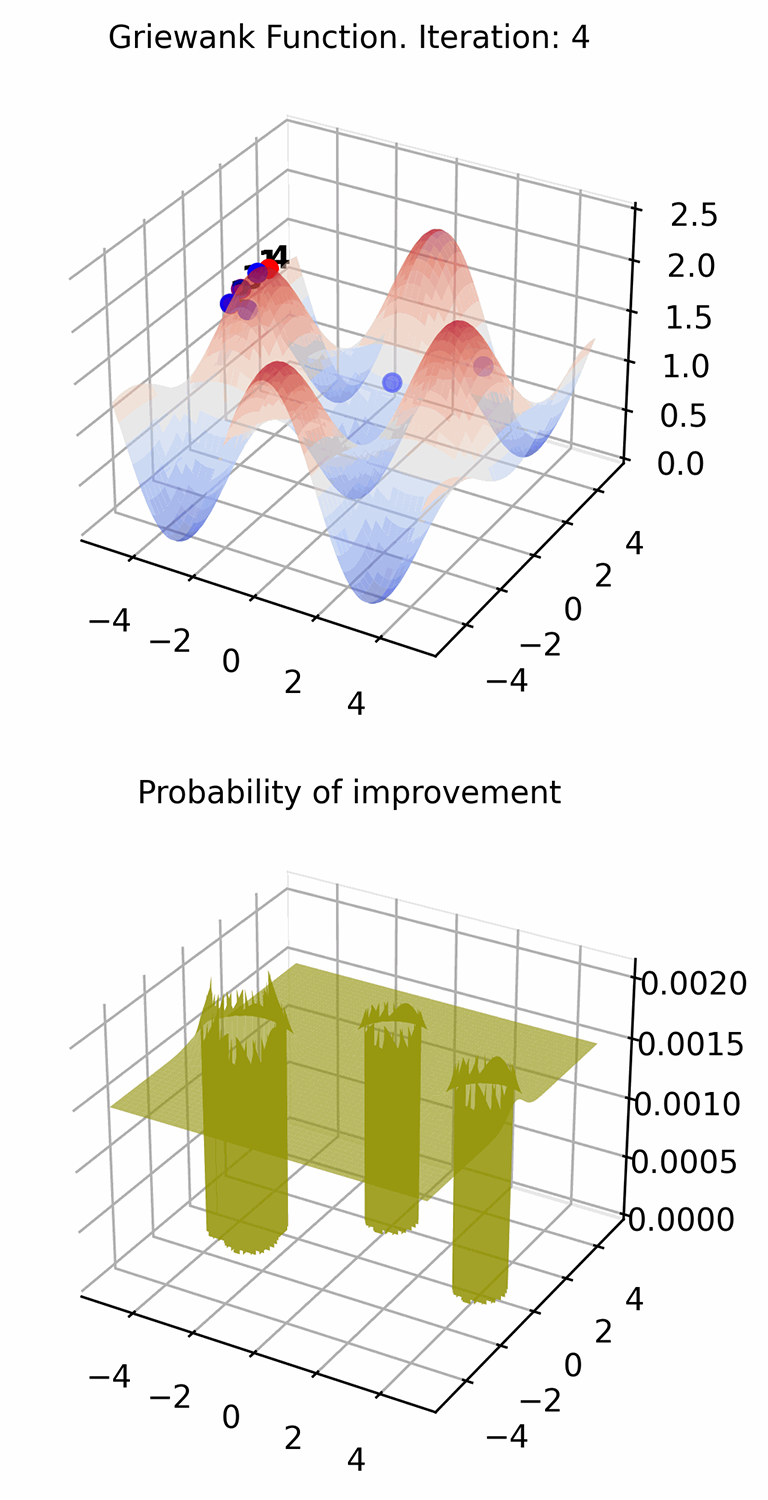}
			\caption{\small Step 13: 1st optimum}
			\label{fig:2d_pi_4}
		\end{subfigure}
		\begin{subfigure}[t]{0.49\textwidth}
			\centering
			\includegraphics[width=\textwidth]{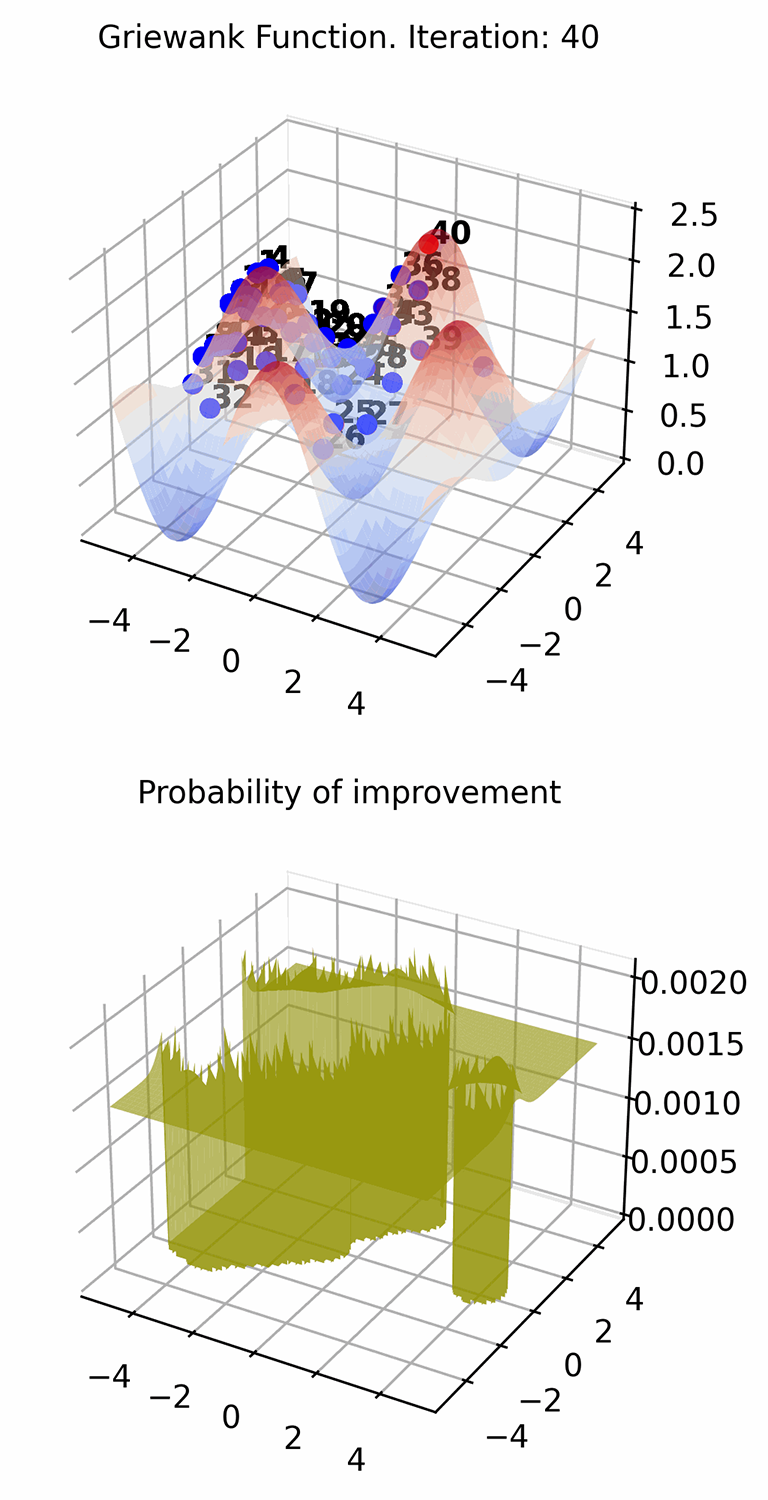}
			\caption{\small Final result}
			\label{fig:2d_pi_40}
		\end{subfigure}
		\caption{Finding local optima on Griewank function via $ a_{PI} $.}
		\label{fig:2d_griewank}
	\end{minipage}%
	\hspace{0.1in}
	\begin{minipage}{0.6\textwidth}
		\centering
		\begin{subfigure}[t]{0.32\textwidth}
			\centering
			\includegraphics[width=\textwidth]{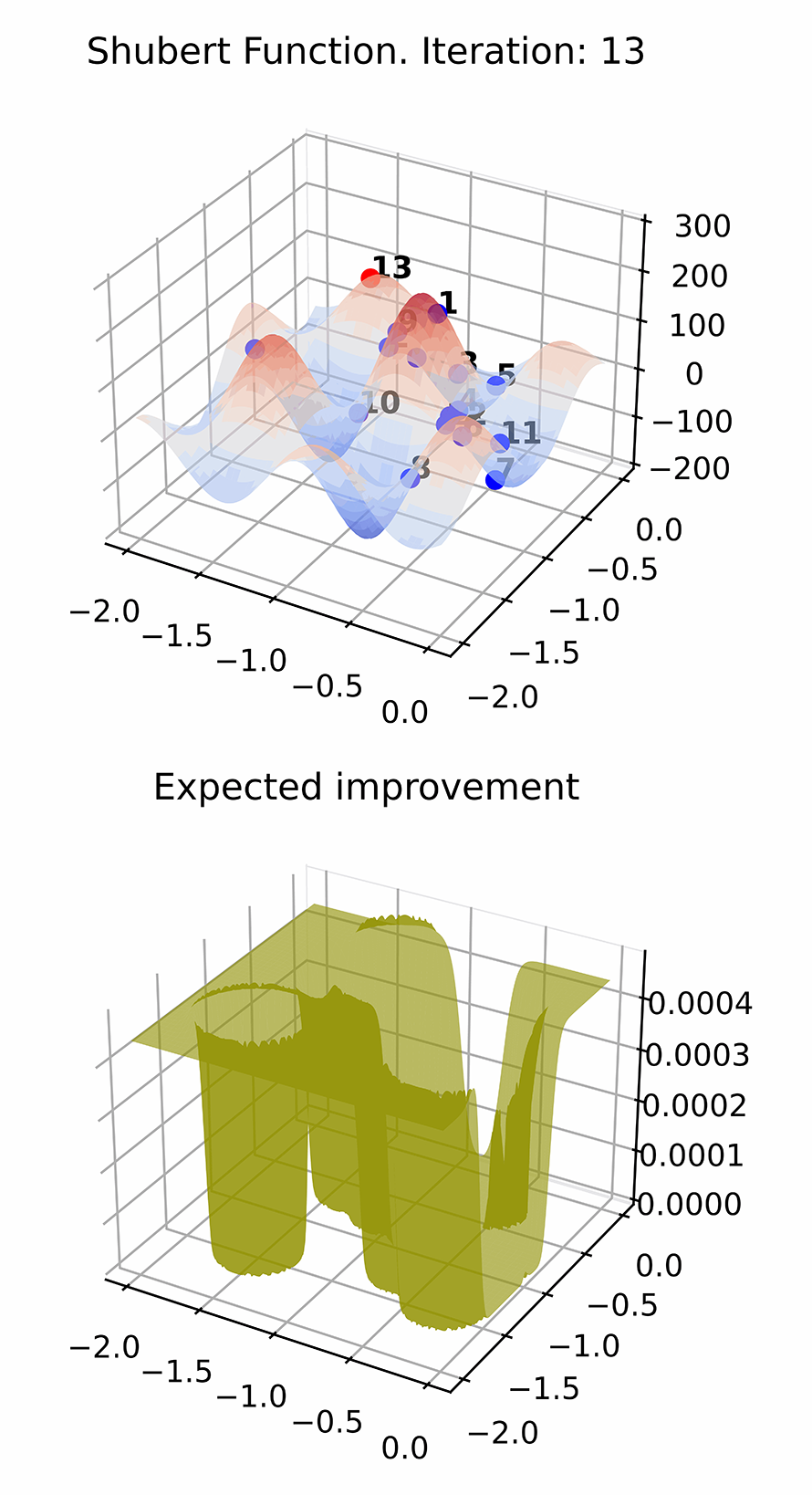}
			\caption{\small Step 13: 1st optimum}
			\label{fig:2d_ei_13}
		\end{subfigure}
		\begin{subfigure}[t]{0.32\textwidth}
			\centering
			\includegraphics[width=\textwidth]{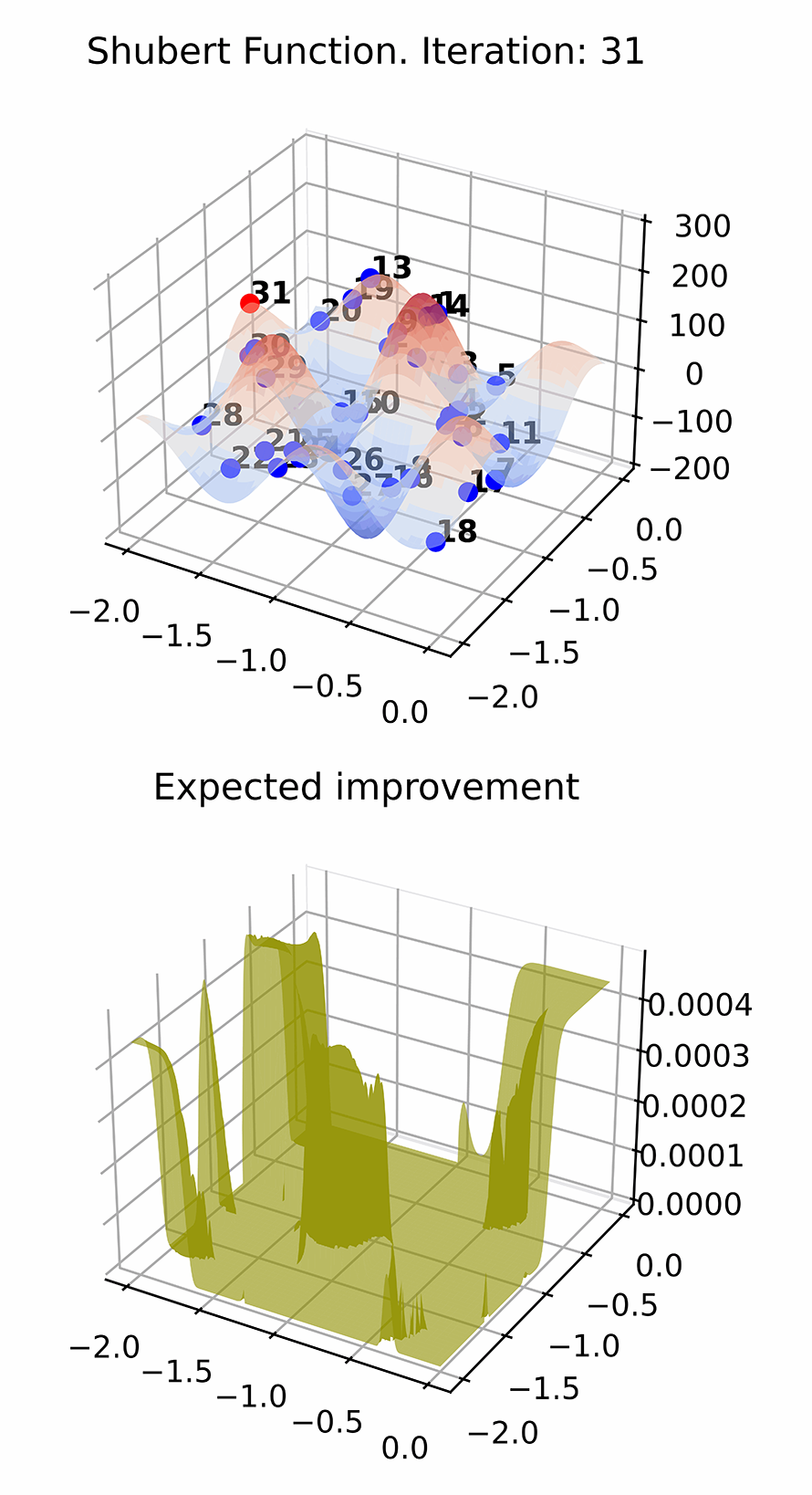}
			\caption{\small Step 31: 2nd optimum}
			\label{fig:2d_ei_31}
		\end{subfigure}
		\begin{subfigure}[t]{0.32\textwidth}
			\centering
			\includegraphics[width=\textwidth]{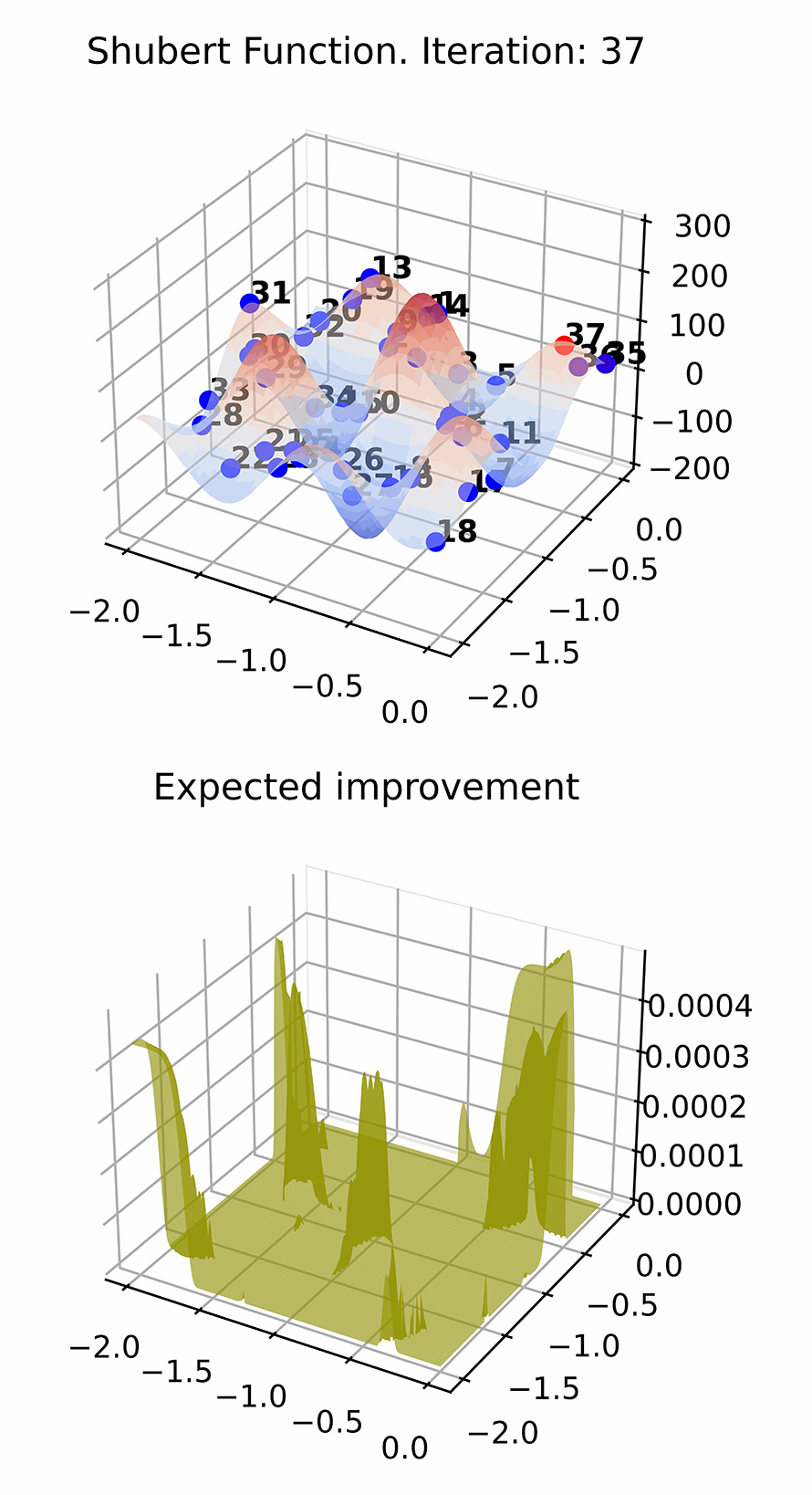}
			\caption{\small Step 37: 3rd optimum}
			\label{fig:2d_ei_37}
		\end{subfigure}
		\caption{Finding local optima on Shubert function via $ a_{EI} $.}
		\label{fig:2d_shubert}
	\end{minipage}
\end{figure*}

\subsection{Multimodal Function}

We implement our framework on two common 2D functions (i.e., the Griewank function and the Shubert function) as benchmark objective functions for this experiment. The results are shown in Fig.~\ref{fig:2d_griewank} and Fig.~\ref{fig:2d_shubert}. The red dots in the figures represent current sampling points, and the blue dots denote previously sampled points. To visualize the sampling order, we show the number of the current step above each dot. Each experiment is tested for 40 iterations.

\subsubsection{Benchmark 1: Griewank Function}

The Griewank function has many widespread local minima, which are regularly distributed. We consider the case where input dimension $ m $ equals 2 and the range of input in each dimension satisfies $ x_i \in [-5, 5] $. The function's graph shows four optima under such conditions in Figure~\ref{fig:2d_griewank}.

We adopt the squared-exponential kernel with $ \alpha $ of 10 and $ l $ of 0.1, and select three random sampling points as priors, which are fixed in each test. Fig.~\ref{fig:2d_pi_40} shows the final result of using joint PI with $ \xi $ of 1 and $ \epsilon $ of 0.1. We manage to reveal a total of two local optima: the first local optimum at step 4 (Fig.~\ref{fig:2d_pi_4}) and the second one at step 40 (Fig.~\ref{fig:2d_pi_40}). The acquisition values for potential sampling points that can be the local optima are higher, while for selected points and their neighbors are zero. Since all optima on the Griewank function are regularly distributed, the points on the grid have nearly similar acquisition values.

\subsubsection{Benchmark 2: Shubert Function} 

As another commonly-used function in optimization problem analysis, the Shubert function is a 2D multimodal function with multiple local and global optima. This function is often evaluated on the square $ x_1,x_2 \in [-10,10] $. To allow for easier viewing, we restrict the domain to $ x_1, x_2 \in [-2, 0] $, presented in Figure~\ref{fig:2d_shubert}.

We use designed joint EI as AF with $ \xi $ of 0 and $ \epsilon $ of 0.1. As shown in Fig.~\ref{fig:2d_shubert}, our framework can find the most local optima over time. We list several key steps where local optima are found in Fig.~\ref{fig:2d_ei_13}, Fig.~\ref{fig:2d_ei_31}, and Fig.~\ref{fig:2d_ei_37}, which illustrates the effectiveness of our design.

\begin{figure*}[t]
	\centering
	\begin{subfigure}[t]{0.22\textwidth}
		\centering
		\includegraphics[width=\textwidth]{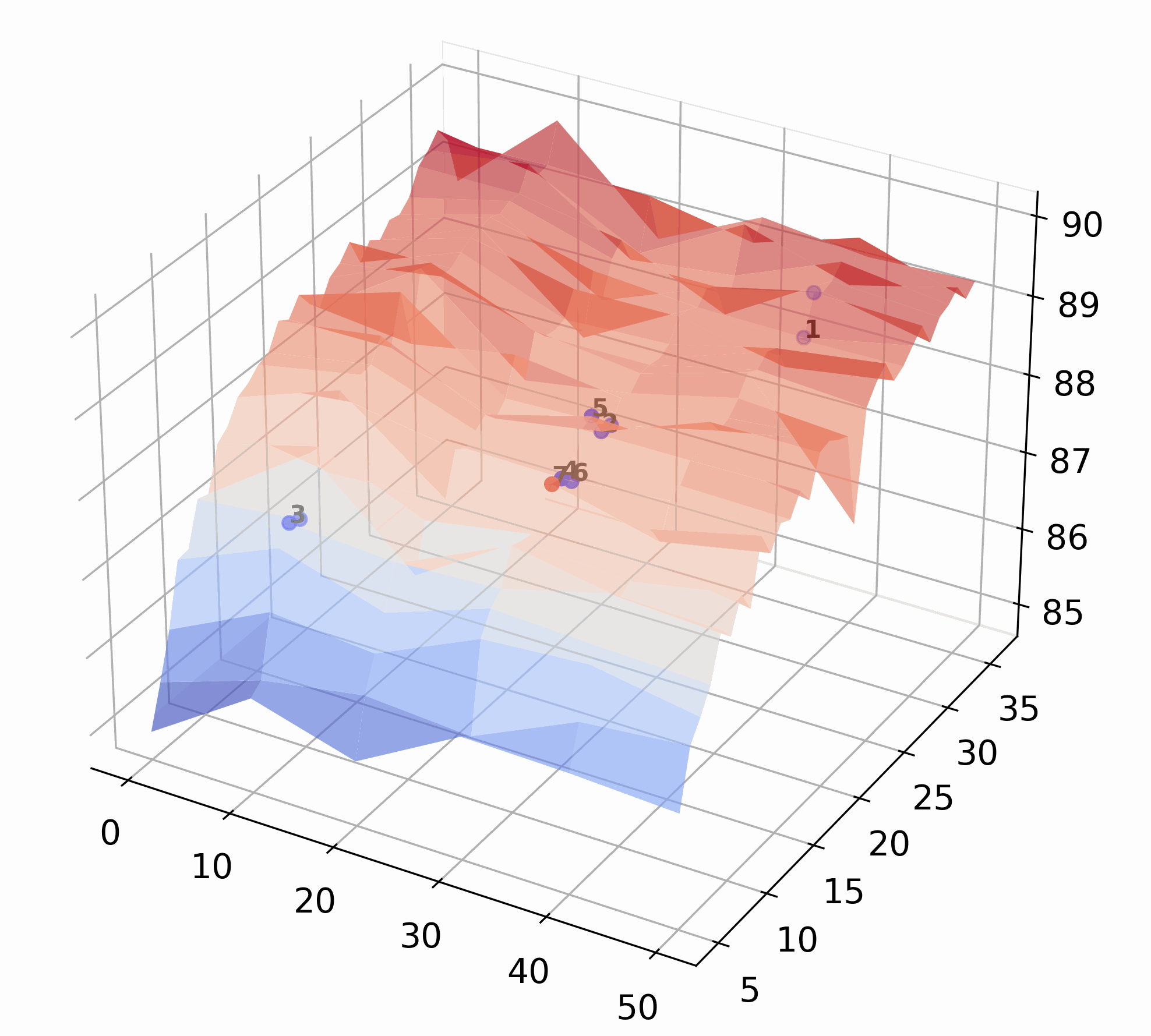}
		\caption{\small Step 7: 1st optimum}
		\label{fig:cifar_7}
	\end{subfigure}
	\begin{subfigure}[t]{0.22\textwidth}
		\centering
		\includegraphics[width=\textwidth]{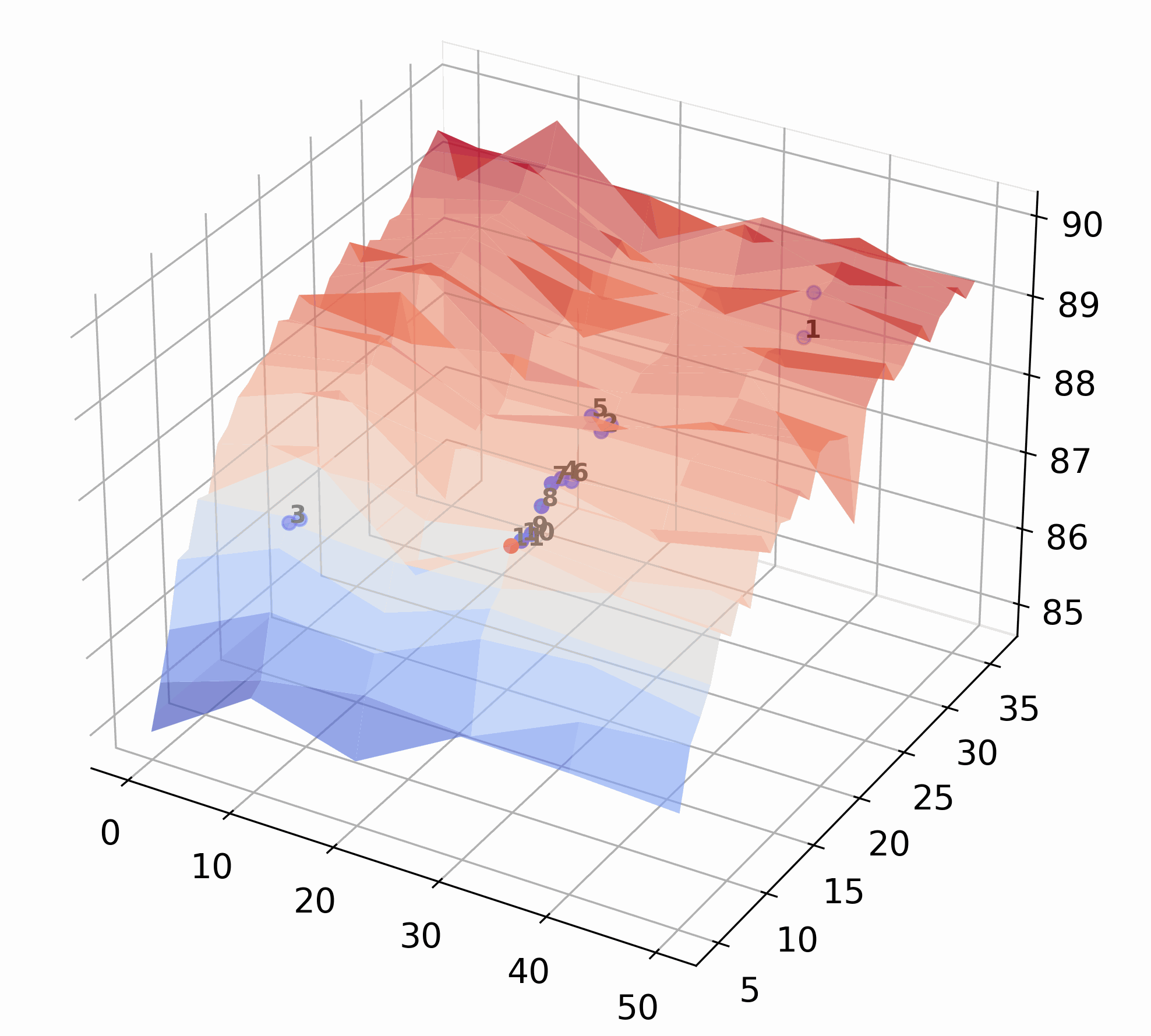}
		\caption{\small Step 11: 2nd optimum}
		\label{fig:cifar_11}
	\end{subfigure}
	\begin{subfigure}[t]{0.22\textwidth}
		\centering
		\includegraphics[width=\textwidth]{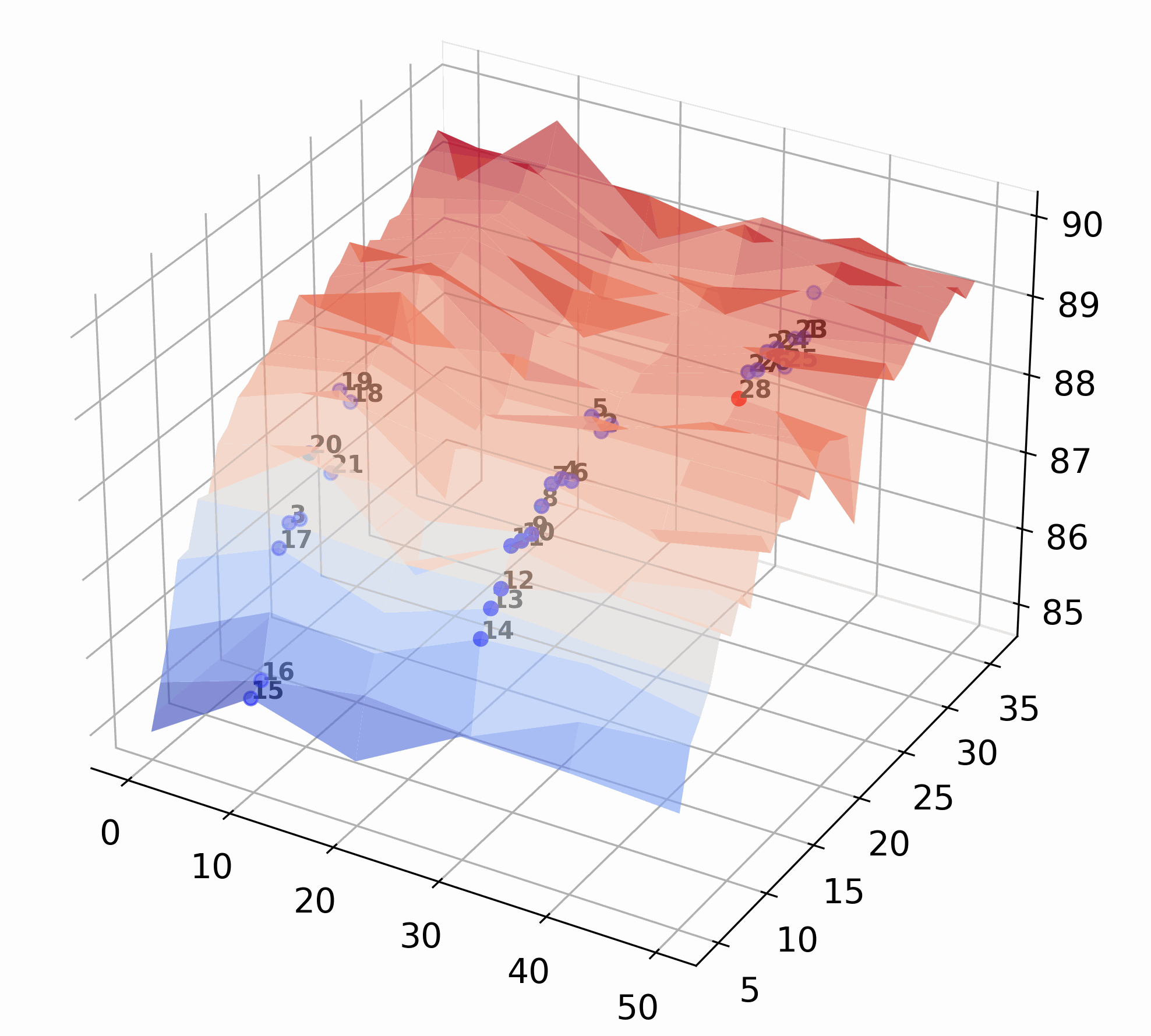}
		\caption{\small Step 28: 3rd optimum}
		\label{fig:cifar_28}
	\end{subfigure}
	\begin{subfigure}[t]{0.22\textwidth}
		\centering
		\includegraphics[width=\textwidth]{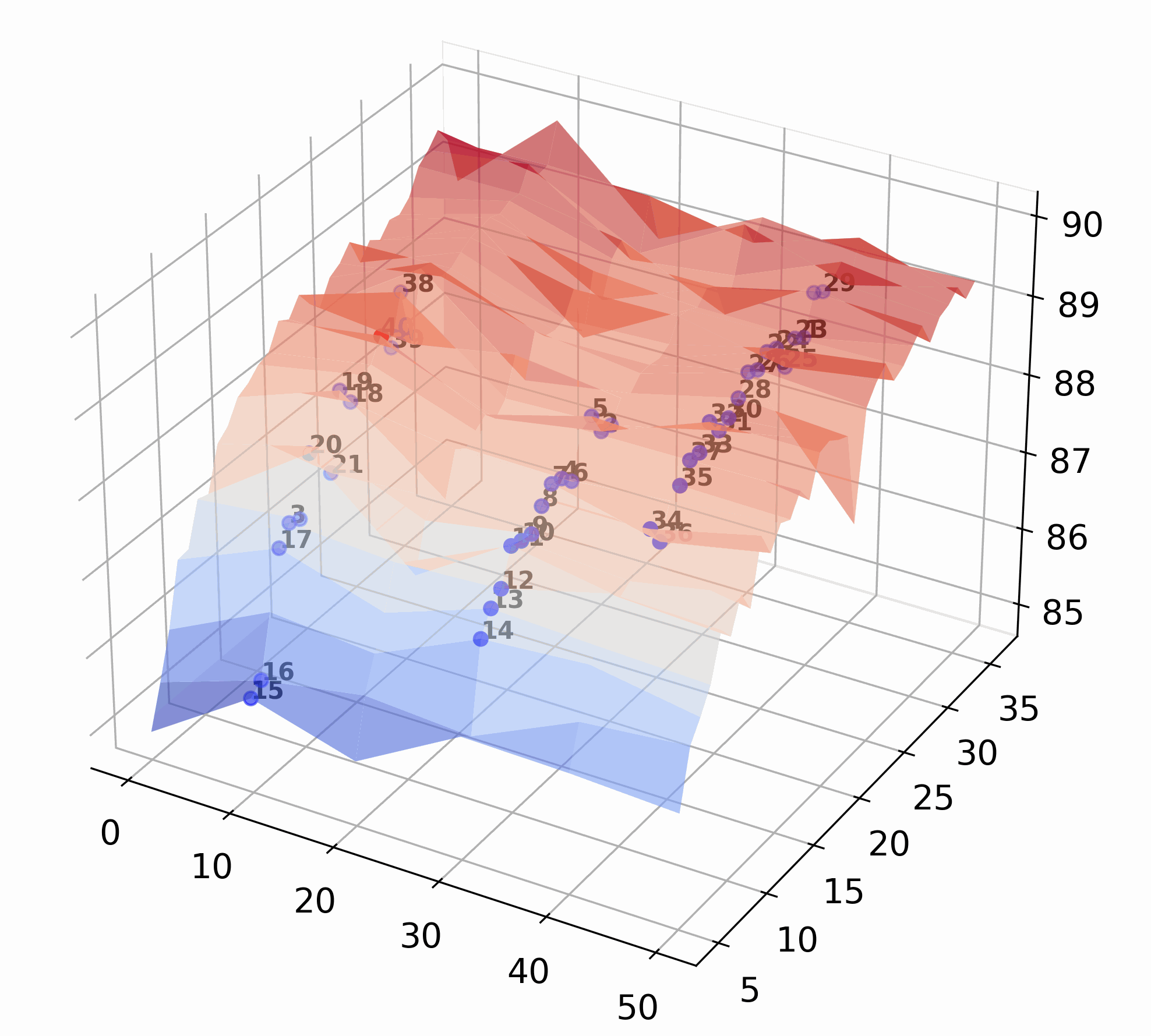}
		\caption{\small Final result}
		\label{fig:cifar_40}
	\end{subfigure}
	\caption{Searching for local optima with $ a_{EI} $ on CIFAR-10 image classification model of backbone ResNeXt.}
	\label{fig:2d_cifar}
\end{figure*}

\subsection{Hyperparameter Tuning in Classification Task}

Adopting the best hyperparameters, even the known ones, are not always feasible in many scenarios where we have additional hardware or runtime restrictions, especially for some expensive machine learning tasks. In these situations, we leverage the local optimal hyperparameters as a compromise. In this part, as a simulation of mentioned cases, we aim to verify our algorithm's capability for finding local optimum hyperparameter combinations on a real-world image classification problem.

\subsubsection{Dataset}

The dataset we use for the image classification task is CIFAR-10, a standard dataset consisting of 60000 32 $ \times $ 32 color images in 10 mutually exclusive classes with 50000 training images and 10000 test images. To fully utilize the dataset, we did necessary preprocessing, including reshaping and padding, with no loss of the information.

\subsubsection{Modeling and Experimentation}

We build the model using ResNeXt~\cite{xie2017aggregated} as the backbone. Among hyperparameters introduced by ResNeXt, we mainly focus on the \textit{layer} of the network and the \textit{width} of the subblock and leave the \textit{cardinality} fixed. Based on this, we create an image classification model as the objective function, where the inputs are the layer and width and output is the validation accuracy. Since training online introduces latent nuance in searching local optima every step, we test our framework to determine local optimum hyperparameter combinations in an offline manner, i.e., we apply our algorithm to a trained model with all the samples ready.

In the experiment, we set cardinality to 4 and vary width from 5 to 35. Besides, we consider 6 layer candidates, which are 29, 38, 50, 68, 86, and 101. Using an interval of 10, we map the 6 layer candidates into a range from 0 to 50 as approximated representations. As shown in Fig.~\ref{fig:2d_cifar}, the accuracy increases with width and layer projection values. However, such an increase is not monotonic, and we can spot many local optimum solutions. Since the mean of all results is around 87.0, we set the $ \xi $ as 87 and $ \epsilon $ as 0.01. We adopt squared-exponential kernel, with $ \alpha $ of 5 and $ l $ of 2.

Fig.~\ref{fig:2d_cifar} exemplifies some steps in which our method finds local optima. Compared to the width, the layer in this scenario impacts more on output accuracy. Therefore, the algorithm prioritizes to explore the local optima of different widths under a specific layer and then searches elsewhere for different layers. For instance, in Fig.~\ref{fig:cifar_11} and Fig.~\ref{fig:cifar_28}, we spend multiple iterations searching for the optima with the layer of 68 before switching to explore that with the layer of 86.

\begin{figure*}[ht!]
	\centering
	\begin{subfigure}[t]{0.28\textwidth}
		\centering
		\includegraphics[width=\textwidth]{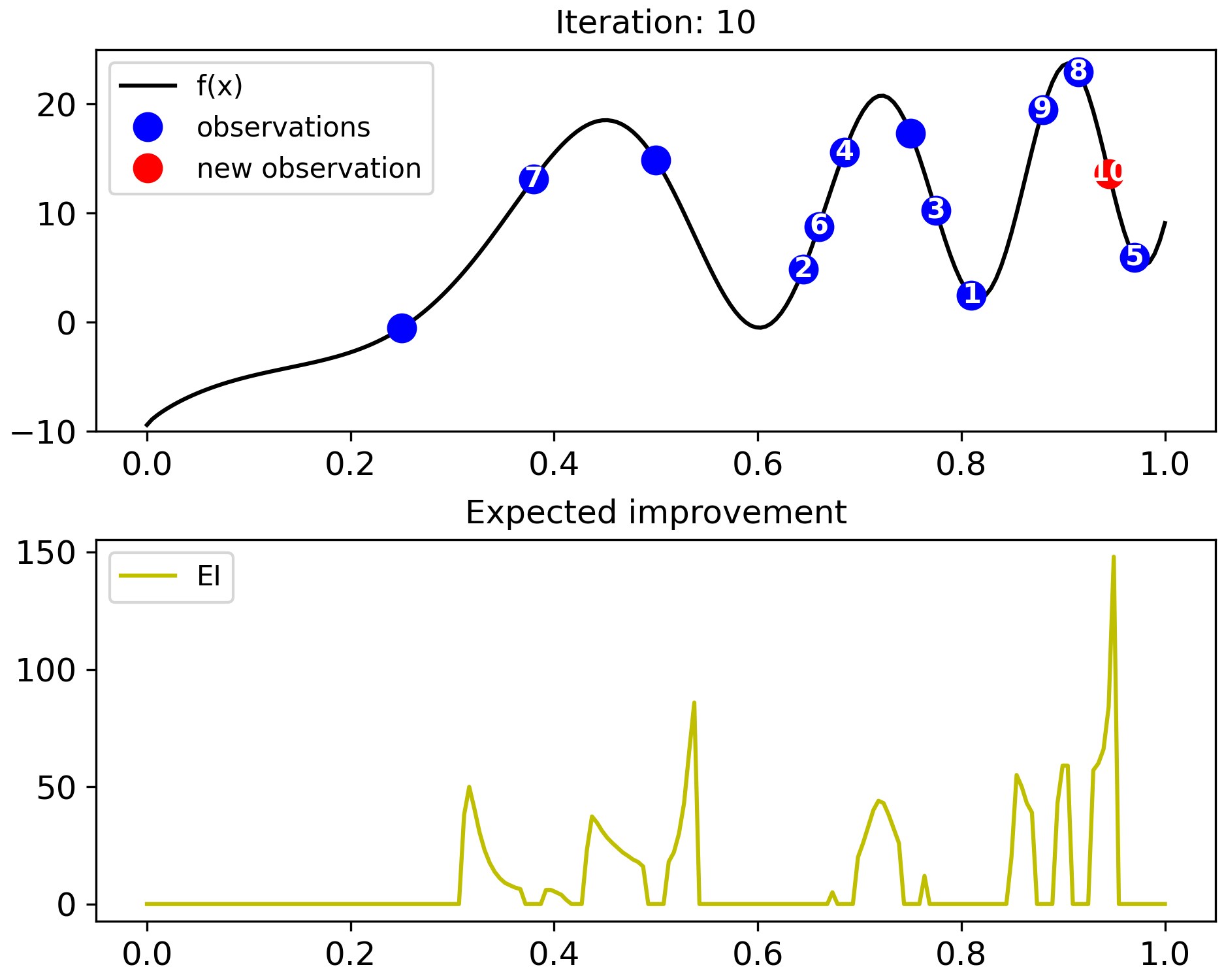}
		\caption{\small Posterior only}
		\label{fig:abla_pred}
	\end{subfigure}
	\begin{subfigure}[t]{0.28\textwidth}
		\centering
		\includegraphics[width=\textwidth]{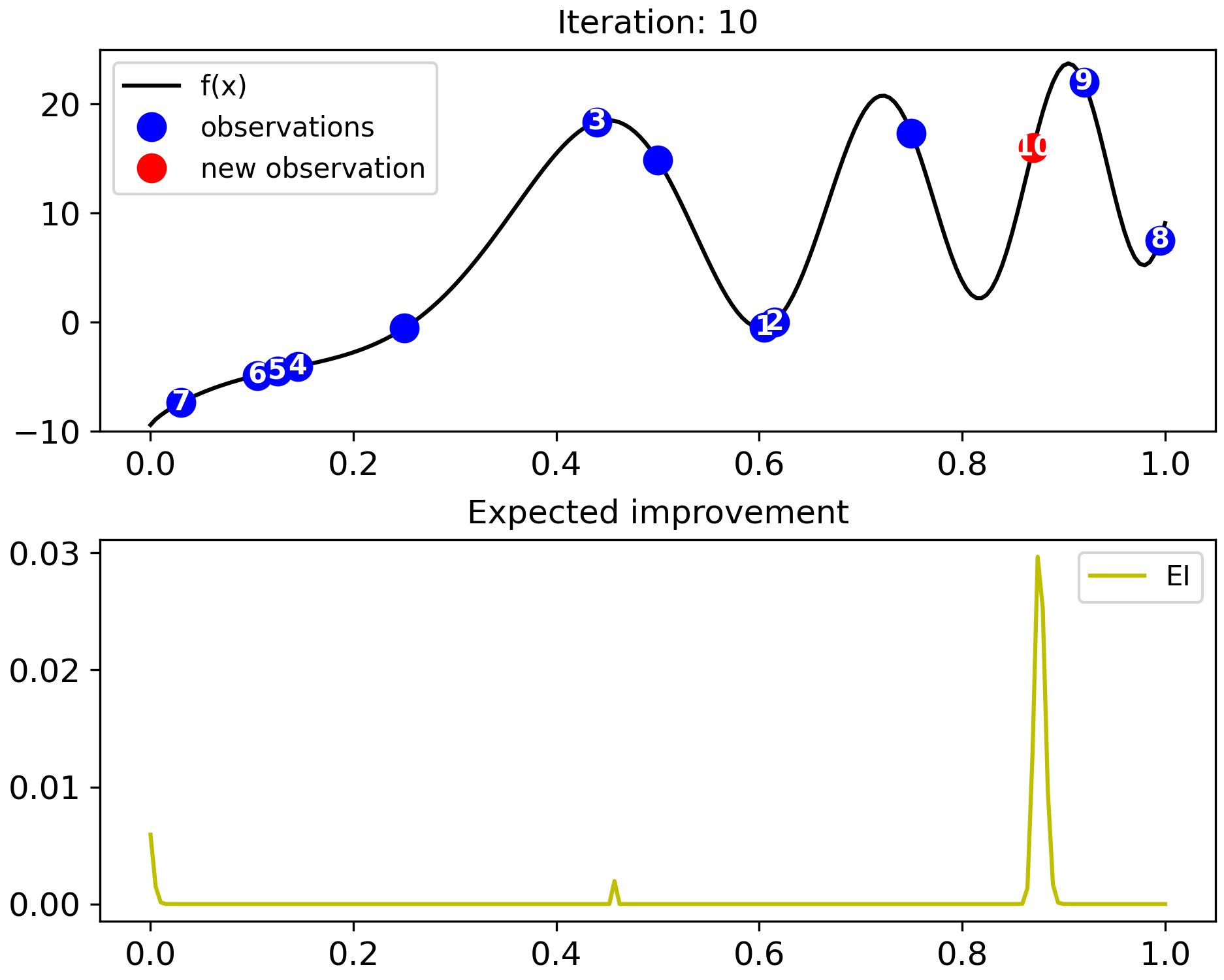}
		\caption{\small First-order derivative only}
		\label{fig:abla_deriv}
	\end{subfigure}
	\begin{subfigure}[t]{0.28\textwidth}
		\centering
		\includegraphics[width=\textwidth]{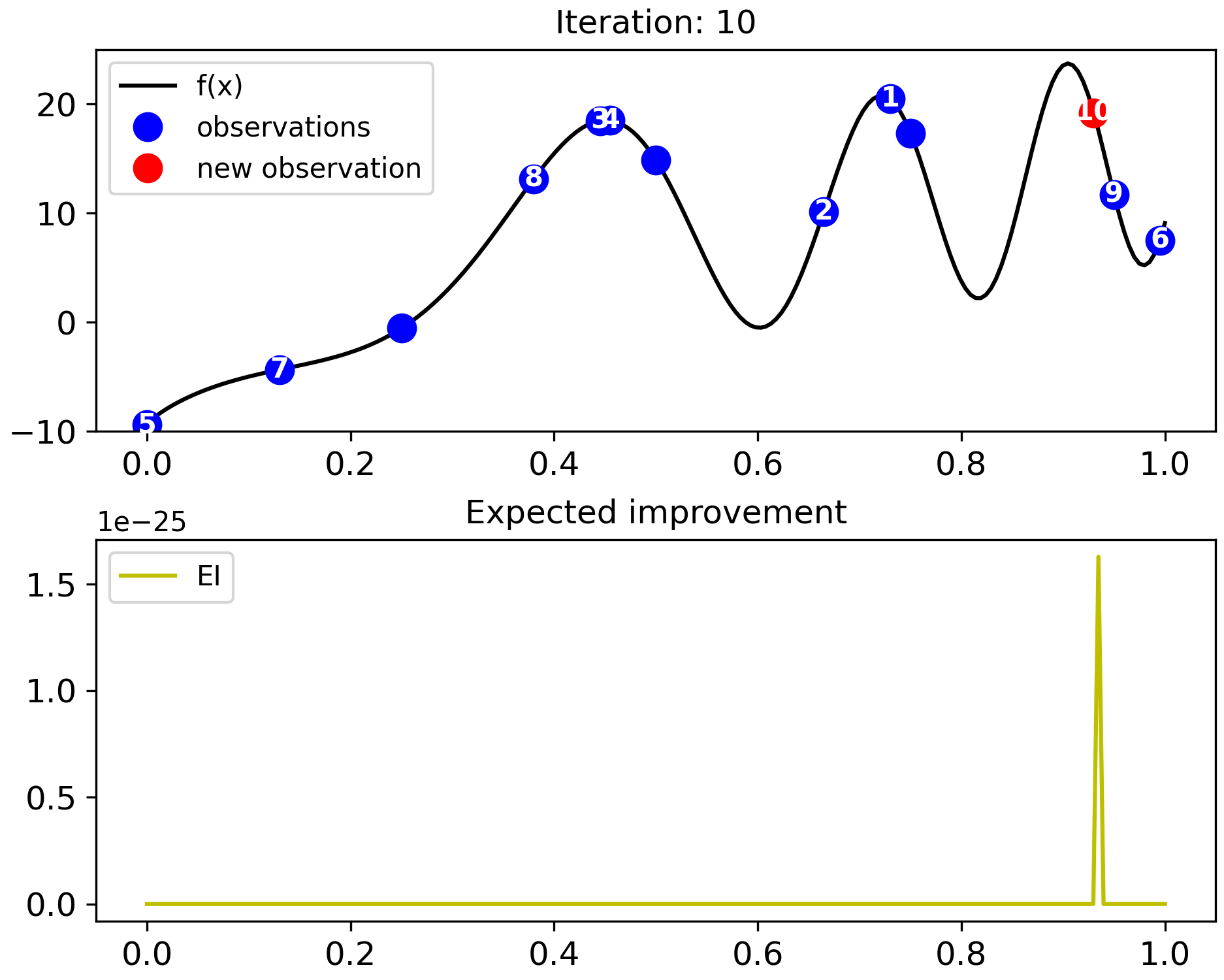}
		\caption{\small Our method (posterior$ + $derivative)}
		\label{fig:abla_ours}
	\end{subfigure}
	\caption{Ablations regarding the Gaussian posterior of the objective (standard BO) or its first-order derivative only.}
	\label{fig:abla}
\end{figure*}

\subsection{Comparison with other baselines}

\begin{table*}
	\caption{Comparison between proposed framework with other baselines in finding the local and global optimum solutions.}
	\centering
	\begin{tabular}{lccccccccc}
		\toprule
		\multirow{2}{*}{Baselines} & \multicolumn{2}{c}{1: Global maxima} & \multicolumn{2}{c}{2: 2nd large local maxima} & \multicolumn{2}{c}{3: 3rd large local maxima} & \multicolumn{3}{c}{Average Distance every 30 steps}\\
		 & Step & Average Distance & Step & Average Distance & Step & Average Distance & Step 30 & Step 60 & Step 90 \\
		\midrule
		Joint PI & 78 & 0.062 & 36 & 0.040 & 12 & 0.034 & 0.043 & 0.068 & 0.059\\
		Joint EI & 80 & 0.061 & 22 & 0.055 & 6 & 0.022 & 0.055 & 0.065 & 0.057 \\
		MPD & 20 & 0.092 & 37 & 0.109 & 47 & 0.089 & 0.128 & 0.080 & 0.081 \\
		GIBO & 25 & 0.151 & 31 & 0.135 & 60 & 0.131 & 0.139 & 0.131 & 0.134 \\
		Vanilla PI & 28 & 0.170 & -- & -- & 54 & 0.209 & 0.204 & 0.197 & 0.288 \\
		Vanilla EI & 18 & 0.090 & 66 & 0.174 & 81 & 0.187 & 0.135 & 0.176 & 0.179 \\
		\bottomrule
	\end{tabular}
	\label{tab:cmp}
\end{table*}

In this section, we compare our method with several selected baselines algorithms: (1) MPD \cite{nguyen2022local}, which uses maximum look-ahead descent probability; and (2) GIBO \cite{muller2021local}, which performs local search by minimizing the trace of the posterior covariance of the gradient only; (3) vanilla BO with PI; (4) vanilla BO with EI. We use a 1D multimodal synthetic function of which the inputs range from 0 to 1 with 200 sampling instances to verify the methods' capability of finding the local and global maxima. Each run has a budget of 100 steps. We show the results in Table~\ref{tab:cmp}, where our framework can locate the objective local maxima faster than other baseline methods. Given a specific threshold, our framework will prioritize locating the nearby local maximum solutions. Besides, compared to the model with the joint probability of improvement, the one with joint expected improvement has better performance owing to its awareness of the potential improvement amount. 

We also compute the average distance by adding the distance between new observation $ \bm{x}_i $ and ground truth $ \bm{x}^* $ of each step $ i $ and taking the average based on the number of already-taken steps $ n $ , given by $ \frac{\sum_i^n \bm{x}_i-\bm{x}^*}{n} $. For every new observation, the distance is recognized as the nearest ground truth $ \bm{x}^* $ to $ \bm{x}_i $. According to Table \ref{tab:cmp}, our framework achieves lower regret than other baseline methods, especially in deciding the local optimum solutions in experiments 2 and 3, demonstrating the effectiveness of finding local optima by adopting our approach.

\subsection{Scalability Experiment}

To demonstrate our method's scalability, we experiment with 3D Griewank functions for locating local/global optimal solutions. Due to the visualization limitation of high-dimensional functions, we post several steps of detected local maxima in Table~\ref{tab:3d_griewank}, where we show that the distance to ground truth becomes smaller in the process for three maxima. We run the experiment for 300 steps in total, with a minimum sampling distance of 0.1. As seen from the table, our algorithm can be easily extended to cope with optimization problems in higher dimensions following the same procedure as processing lower-dimensional multimodal functions. Specifically, for the 3D Griewank function shown in Table~\ref{tab:3d_griewank}, we successfully locate the solutions at steps 29, 169, and 245 with the distances to ground truth maxima as 0.024, 0, and 0.1, respectively.

\begin{table}
	\caption{Finding maxima on 3D Griewank function.}
	\centering
	\setlength{\tabcolsep}{3pt}
	\begin{tabular}{ccccccc}
		\toprule
		Step & Maxima & $ x_1 $ & $ x_2 $ & $ x_3 $ & $ f(x_1,x_2,x_3) $ & Distance \\
		\midrule
		27 & & -3.6 & 0.9 & 0.0 & 1.725 & 1.030 \\
		28 & & -3.0 & -0.8 & 0.2 & 1.833 & 0.831 \\ 
		\textbf{29} & 1 & -3.0 & -0.2 & 0.0 & \textbf{1.982} & \textbf{0.224}\\
		30 & & -3.4 & -0.2 & 0.0 & 1.907 & 0.361 \\
		31 & & -2.9 & 0.2 & -0.5 & 1.924 & 0.574 \\
		\midrule
		167 & & 0.4 & -4.2 & 0.0 & 1.912 & 0.447 \\
		168 & & -0.1 & -4.4 & 0.2 & 1.993 & 0.224 \\
		\textbf{169} & 2 & 0.0 & -4.4 & 0.0 & \textbf{2.004} & \textbf{0.000} \\
		170 & & 0.0 & -4.3 & -0.3 & 1.985 & 0.316 \\
		171 & & -0.3 & -4.4 & 0.2 & 1.953 & 0.361 \\
		\midrule
		243 & & -0.5 & 4.2 & -0.2 & 1.863 & 0.574 \\
		244 & & 0.2 & 4.3 & -0.1 & 1.978 & 0.245 \\
		\textbf{245} & 3 & 0.0 & 4.4 & -0.1 & \textbf{2.003} & \textbf{0.100} \\
		246 & & -0.3 & 4.4 & 0.5 & 1.934 & 0.583 \\
		247 & & -0.1 & 4.5 & -0.3 & 1.984 & 0.332 \\
		\bottomrule
	\end{tabular}
	\label{tab:3d_griewank}
\end{table}

\subsection{Ablation Experiment}

In this experiment, we present the ablation results by removing the first-order derivative and posterior of the objective in AF (which is expected improvement) and run the test on a synthetic function. We present the results at the tenth iteration. As can be seen in Fig.~\ref{fig:abla_pred}, once we remove the first-order derivative condition, the AF will be similar to the standard BO, and the algorithm will start looking for the points to maximize the acquisition value, failing to rapidly recognize other local/global maxima in this multimodal function. The two local maxima are not found during iterations, while our method successfully reveals them. In Fig.~\ref{fig:abla_deriv}, if we only consider the first-order derivative condition in choosing the next sampling point, the sequential search could easily get trapped in minima or stationary points without improving the objective value toward local maxima, such as points 1 and 2 or points 4, 5, and 6. In contrast, our algorithm can explicitly find the local maximum solutions by considering both conditions, given by Fig.~\ref{fig:abla_ours}, where the maxima are points 1 and 4. The ablations indicate that our method can detect the global maximum and local maxima for given objective functions.

\section{Conclusion}
\label{sec:conclusion}

We propose a BO framework for finding a set of local/global optima in multimodal objective functions that are expensive to evaluate. Given a posterior modeled as a Gaussian Process, our method considers the joint distribution of the objective function and its first-order derivative. Evaluations on benchmark functions, image classification tasks, and hyperparameter tuning problems demonstrate the effectiveness of our solution in finding local/global optima.

%% file: appendix.tex
\newpage
\appendix

\section{Proof of Lemma~\ref{lem:dct}}
\label{subsec:proof_lem}

To investigate the theoretical support for interchangeable condition, as shown in the Lemma~\ref{lem:dct}, we have:
\begin{equation*}
	\begin{aligned}
		\frac{\p}{\p t} \mathbb{E}\left[g(t, Y)\right] &= \lim_{h\to 0} \frac1h \bigl( \mathbb{E}\left[g(t+h, Y)\right] - \mathbb{E}\left[g(t, Y)\right] \bigr) \\
		&= \lim_{h\to 0} \mathbb{E}\left[ \frac{g(t+h, Y) - g(t, Y)}{h} \right] \\
		&= \lim_{h\to 0} \mathbb{E}\left[ \frac{\partial}{\partial t} g(\tau(h), Y) \right],
	\end{aligned}
\end{equation*}
where $ \tau(h) \in (t, t+h) $ exists by the mean value theorem. By assumption we have:
\begin{equation*}
	\Bigl| \frac{\partial}{\partial t} g(\tau(h), Y) \Bigr| \leq Z,
\end{equation*}
and thus we can use the dominated convergence theorem to conclude:
\begin{equation*}
	\begin{aligned}
		\frac{\p}{\p t} \mathbb{E}\left[g(t, Y)\right] 
		&= \mathbb{E}\left[ \lim_{h\to 0} \frac{\partial}{\partial t} g(\tau(h), Y) \right] \\
		&= \mathbb{E}\left[ \frac{\partial}{\partial t} g(t, Y) \right].
	\end{aligned}
\end{equation*}
This concludes the proof.

\section{Sensitivity Experiment}

\begin{figure}[h]
	\centering
	\begin{subfigure}[t]{0.15\textwidth}
		\centering
		\includegraphics[width=\textwidth]{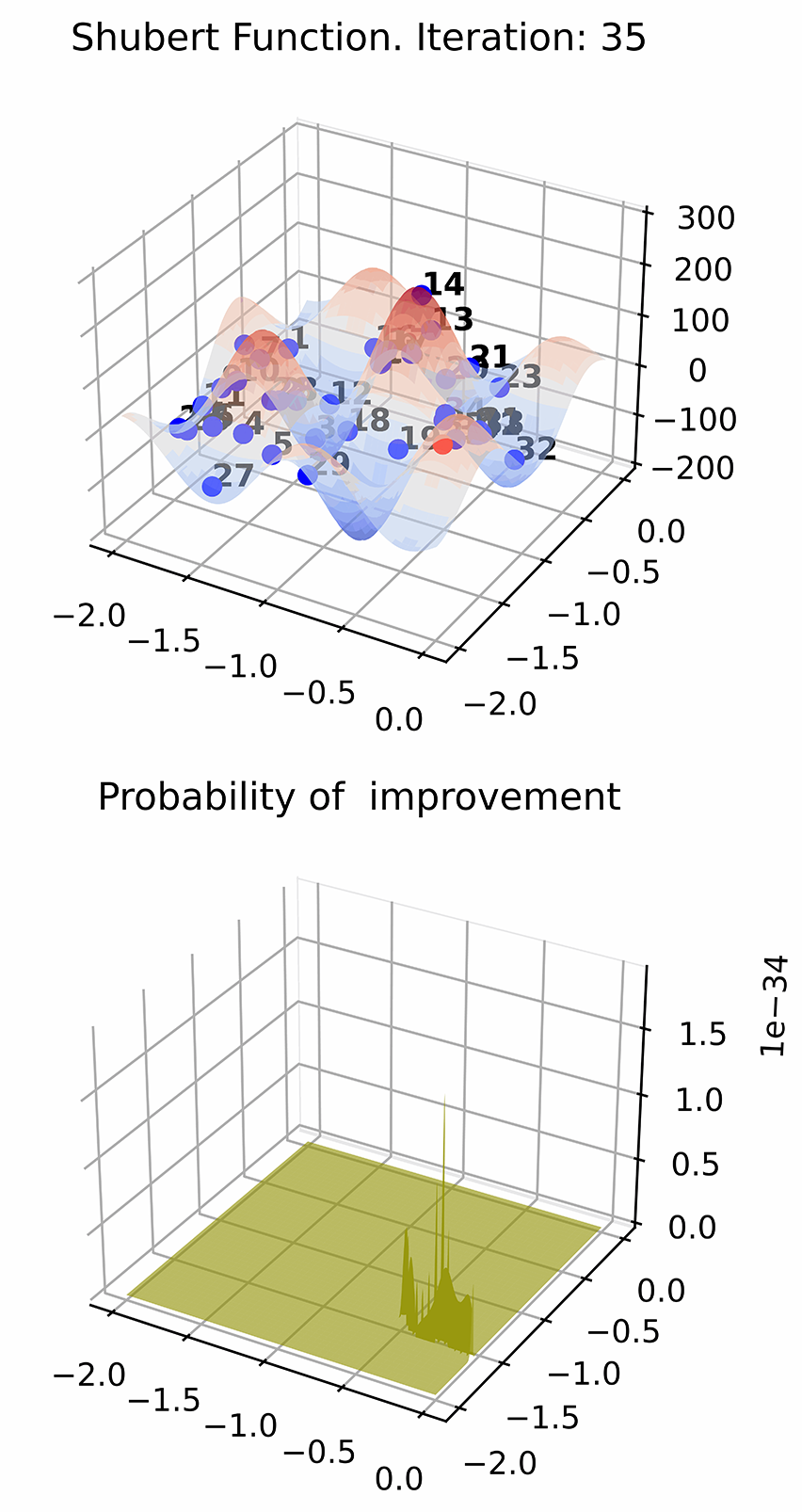}
		\caption{\small $ \alpha=10, l=0.1 $}
		\label{fig:sen_alpha10}
	\end{subfigure}
	\begin{subfigure}[t]{0.15\textwidth}
		\centering
		\includegraphics[width=\textwidth]{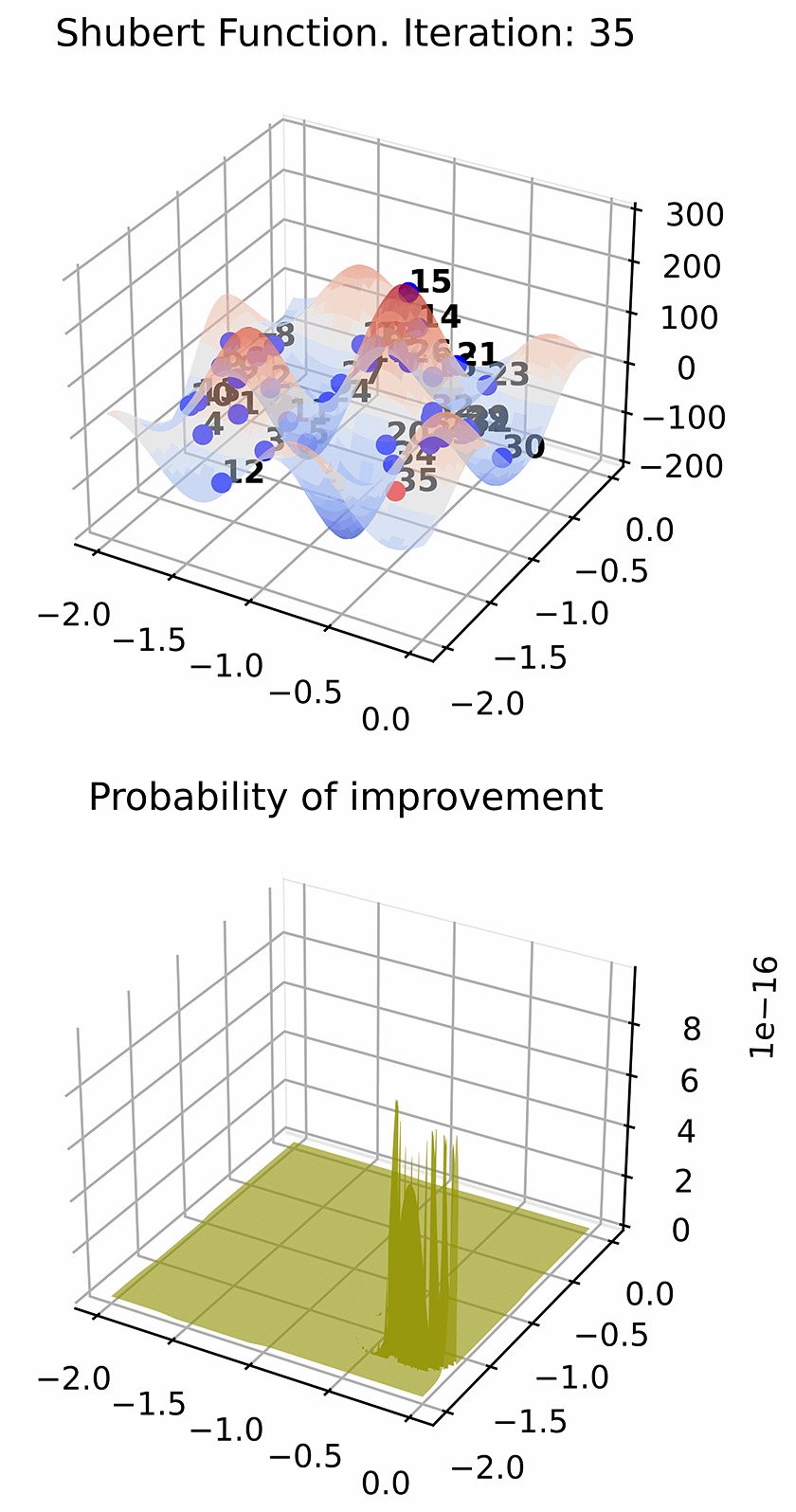}
		\caption{\small $ \alpha=30, l=0.1 $}
		\label{fig:sen_alpha30}
	\end{subfigure}
	\begin{subfigure}[t]{0.15\textwidth}
		\centering
		\includegraphics[width=\textwidth]{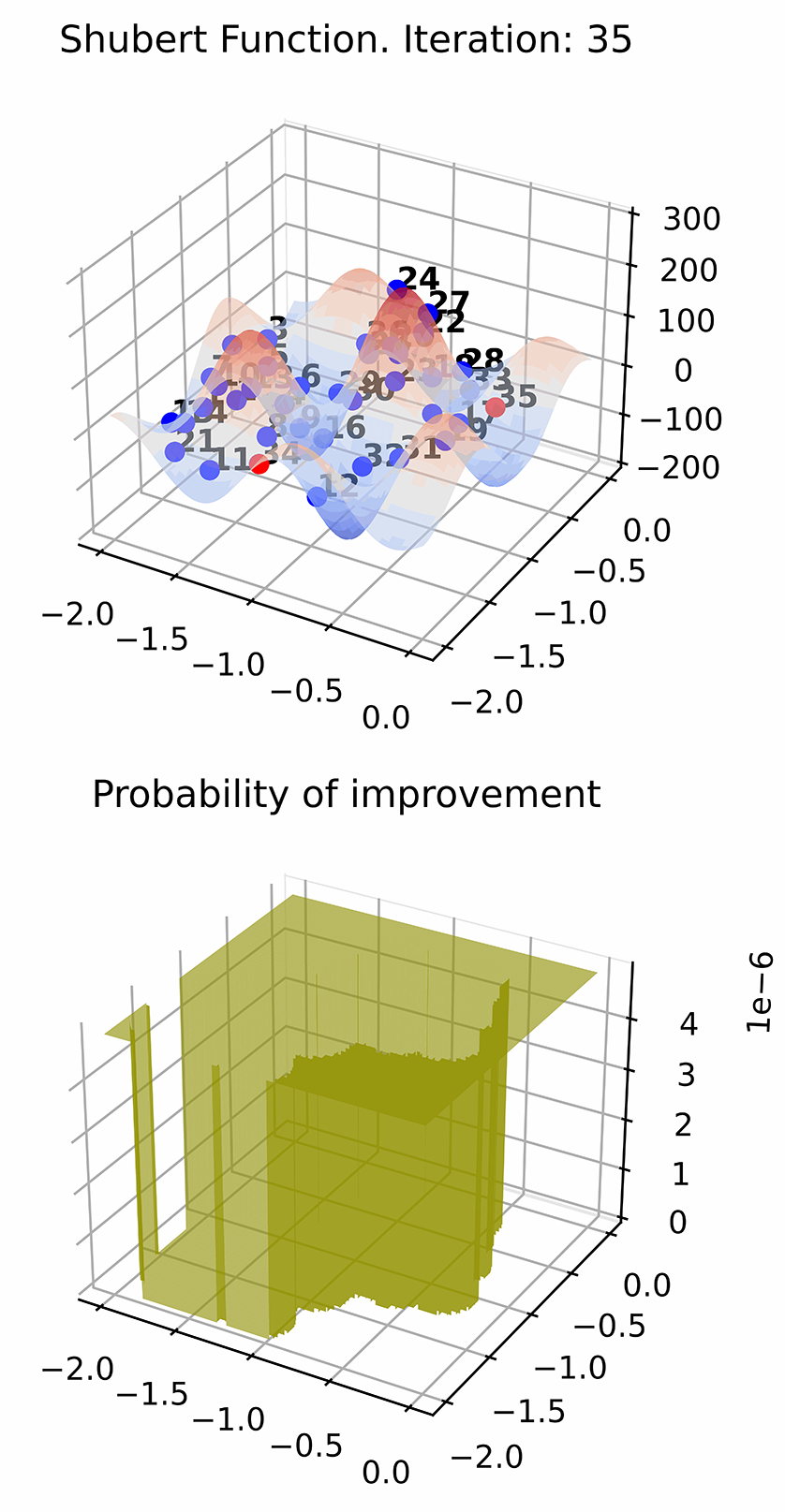}
		\caption{\small $ \alpha=10, l=0.05 $}
		\label{fig:sen_l0.05}
	\end{subfigure}
	\caption{Different scale factors with Shubert function.}
	\label{fig:sen_alpha_l}
\end{figure}

Additionally, we report sensitivities using the Shubert function as the benchmark regarding particular parameters used in our design, including scale factors $ \alpha $ and $ l $ in the squared-exponential kernel, threshold $ \xi $ in AFs, and the sampling distance $ d $. 

{\em Scale factor}: As local optimum solutions are found over time, the AF value will decrease. A fast-decreasing AF value for an objective function with multiple optimum solutions will let the algorithm overlook several unrevealed latent optima. To avoid that, we have to preset the scale factors of the squared-exponential kernel within an appropriate range. The experiments discuss the cases of tuning $ \alpha $ and $ l $. In Fig.~\ref{fig:sen_alpha30}, when we increase the $ \alpha $ from 10 to 30, after 35 iterations, the acquisition value is accordingly larger and thus increases the probability of searching for latent local optima in the indicated area. It is worth noting that over-increasing this parameter will make the sampling position vary too much. Moreover, decreasing $ l $ is another way to raise the acquisition value, as shown in Fig.~\ref{fig:sen_l0.05}, but this will reduce the discrimination of acquisition value of different points.

\begin{figure}[h]
	\centering
	\begin{subfigure}[t]{0.15\textwidth}
		\centering
		\includegraphics[width=\textwidth]{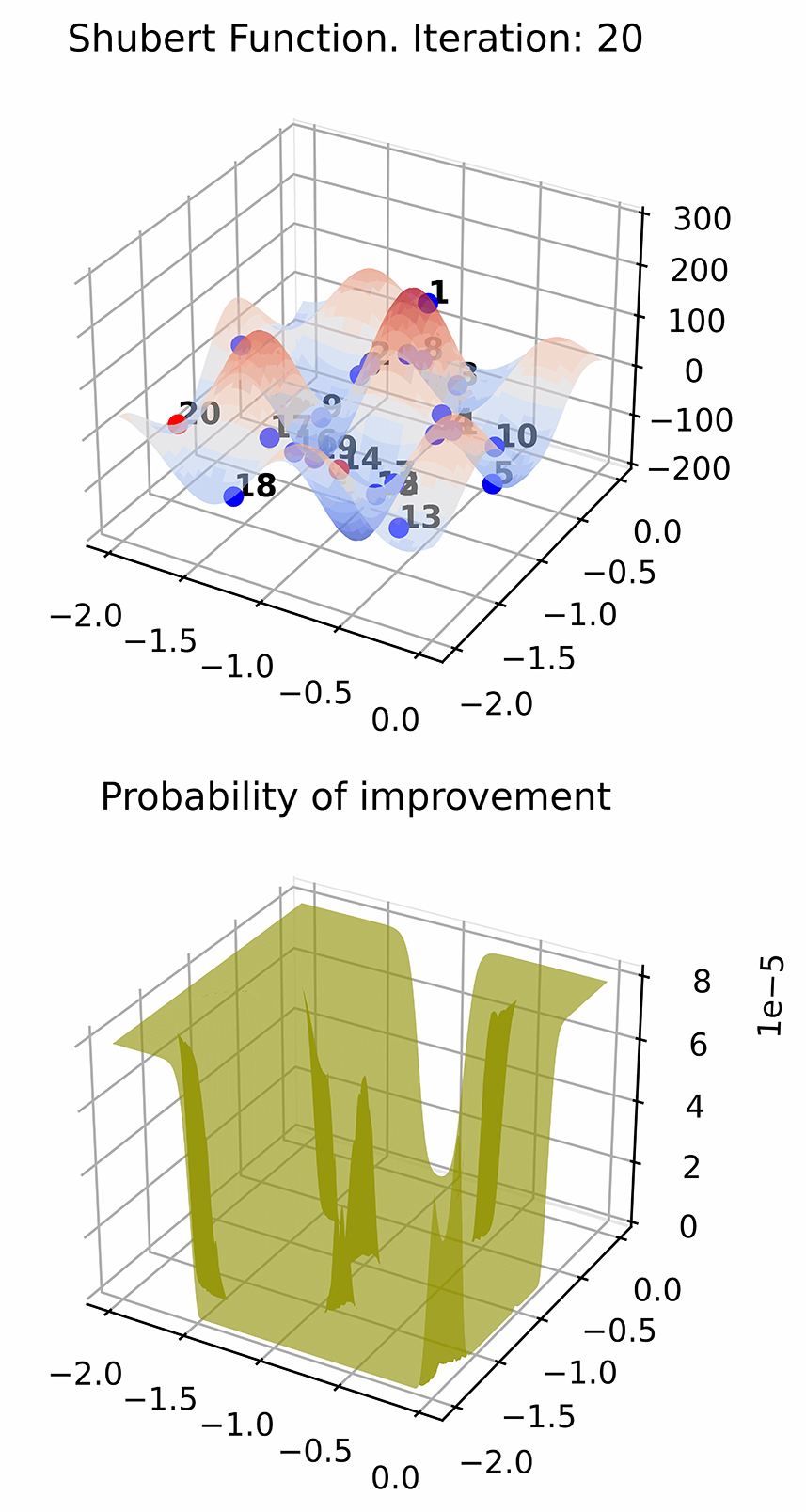}
		\caption{\small $ \xi=0 $}
		\label{fig:sen_t0}
	\end{subfigure}
	\begin{subfigure}[t]{0.15\textwidth}
		\centering
		\includegraphics[width=\textwidth]{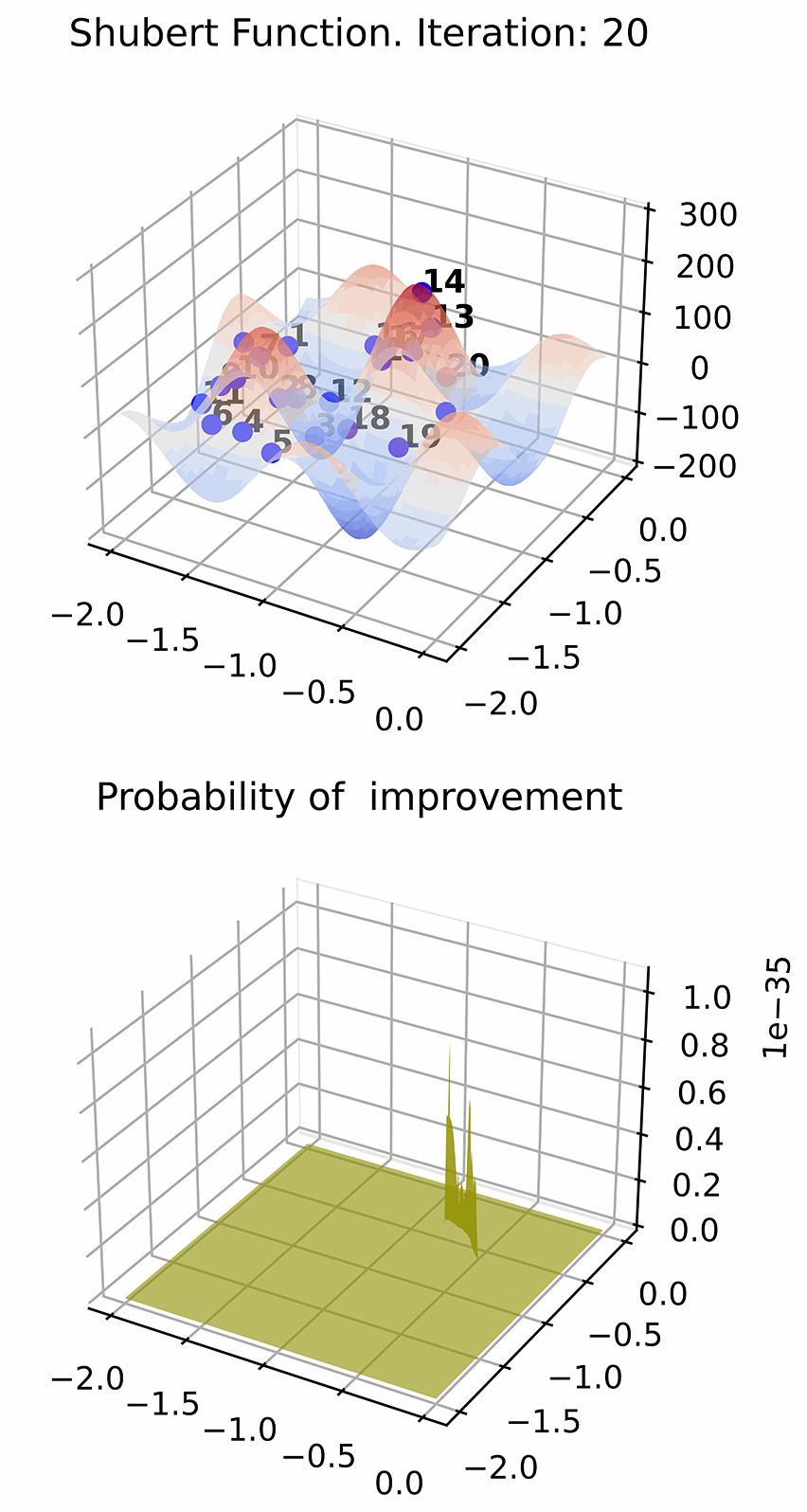}
		\caption{\small $ \xi=40 $}
		\label{fig:sen_t40}
	\end{subfigure}
	\begin{subfigure}[t]{0.15\textwidth}
		\centering
		\includegraphics[width=\textwidth]{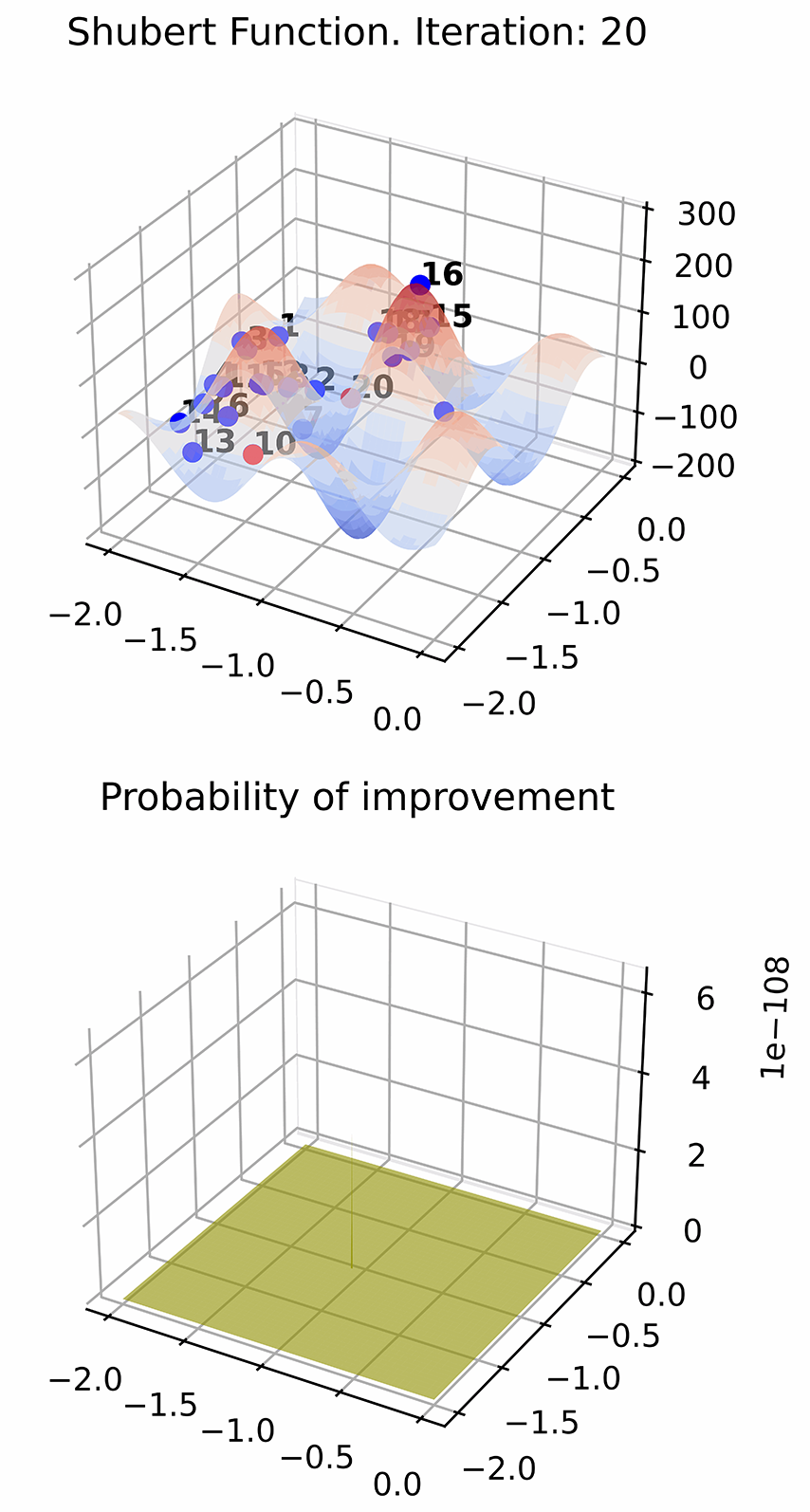}
		\caption{\small $ \xi=80 $}
		\label{fig:sen_t80}
	\end{subfigure}
	\caption{Different thresholds with Shubert function.}
	\label{fig:sen_t}
\end{figure}

{\em Threshold}: As discussed in Corollary~\ref{coro:pi} and \ref{coro:ei}, the threshold $ \xi $ defines the interval where we desire the local optima. After setting $ \xi $, the algorithm will prefer to find the optima higher than the threshold with lower confidence. We present the results of three different threshold settings after 20 iterations in Fig.~\ref{fig:sen_t}. When we increase the threshold, the acquisition function (AF) will exclusively suggest the next latent optima points with the lower value, given Fig.~\ref{fig:sen_t80}, while the one with lower threshold will consider almost all the sampling point with the same probability, as shown in Fig.~\ref{fig:sen_t0}. For functions with many different local optima, over-increasing the threshold may reduce the algorithm's capability of finding the optimum solutions accurately.

\begin{figure}[h]
	\centering
	\begin{subfigure}[t]{0.20\textwidth}
		\centering
		\includegraphics[width=\textwidth]{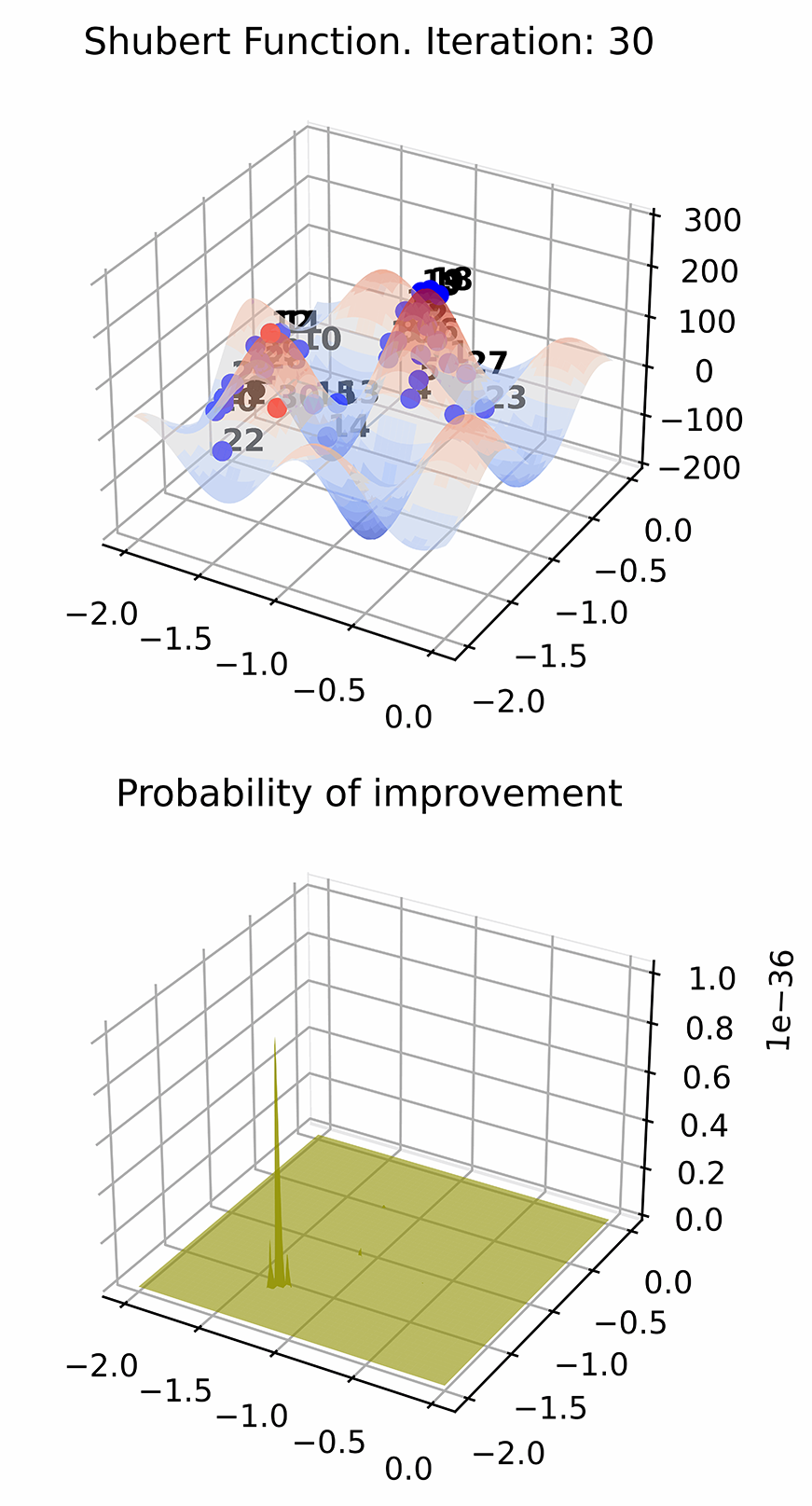}
		\caption{\small $ d=0.05 $}
		\label{fig:sen_d0.05}
	\end{subfigure}
	\begin{subfigure}[t]{0.196\textwidth}
		\centering
		\includegraphics[width=\textwidth]{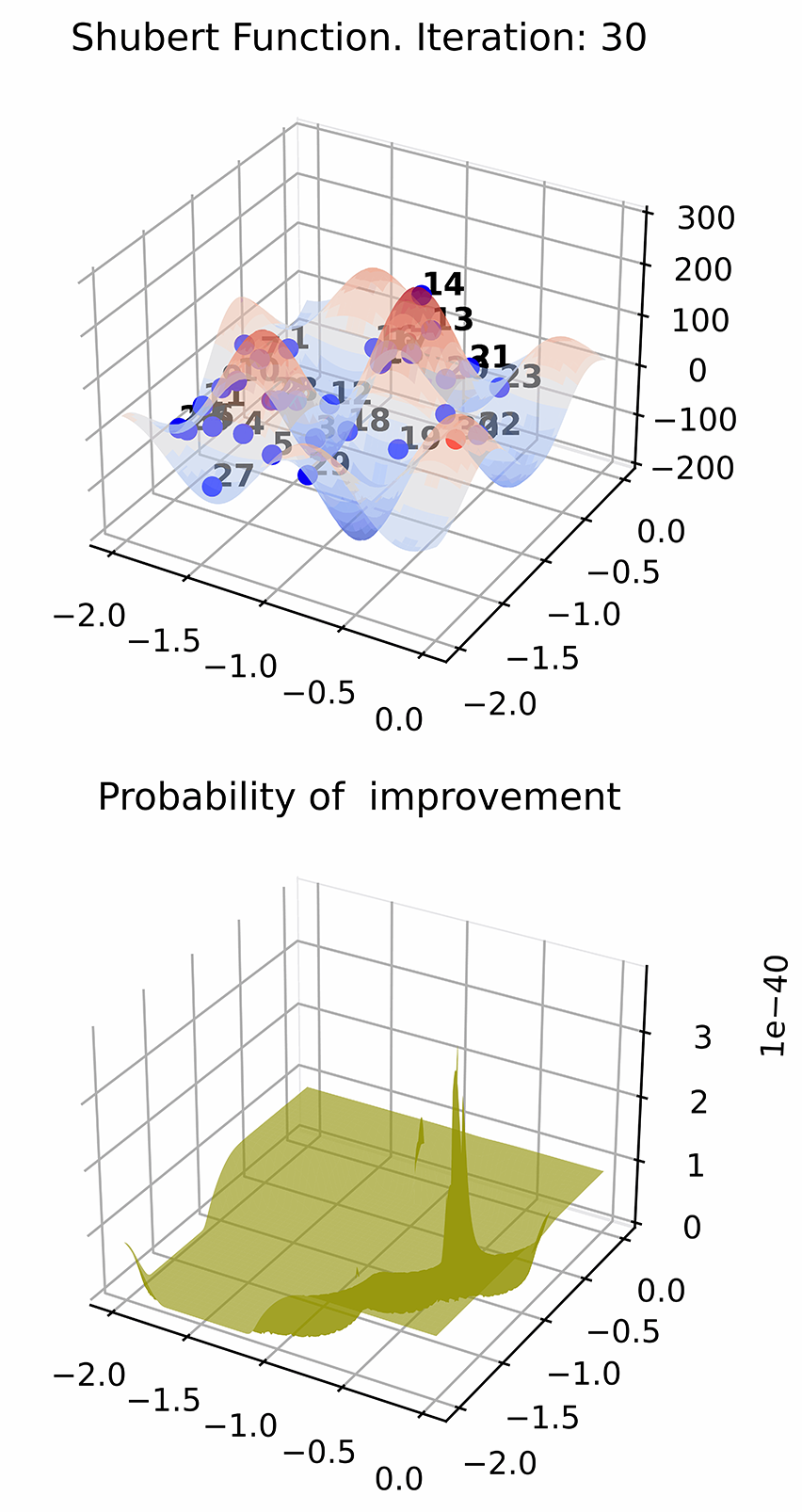}
		\caption{\small $ d=0.8 $}
		\label{fig:sen_d0.8}
	\end{subfigure}
	\caption{Different sampling distances with Shubert function.}
	\label{fig:sen_d}
\end{figure}

{\em Sampling Distance}: The sampling distance $ d $ defines the smallest Euclidean distance between two selected points. We introduce this factor to avoid over-searching within a particular area and improve the algorithm's generalization capability in finding local optima within a small number of iterations. For some objective functions with only a few optimum solutions, reducing $ d $ will help to locate an exact optimum point. However, if an objective function has multiple local optima, such as the Shubert function, we need to increase $ d $ to speed up searching for all possible optimum solutions in the given input domain, as shown in Fig.~\ref{fig:sen_d}. Additionally, as seen in the figure, the larger the sampling distance, the faster the acquisition value will decrease.

%% file: 630Mei.bbl
\begin{thebibliography}{10}

\bibitem{acerbi2017practical}
Luigi Acerbi and Wei~Ji Ma, `Practical bayesian optimization for model fitting
  with bayesian adaptive direct search', {\em Advances in neural information
  processing systems}, {\bf 30}, (2017).

\bibitem{ariafar2019admmbo}
Setareh Ariafar, Jaume Coll-Font, Dana~H Brooks, and Jennifer~G Dy, `Admmbo:
  Bayesian optimization with unknown constraints using admm.', {\em J. Mach.
  Learn. Res.}, {\bf 20}(123),  1--26, (2019).

\bibitem{balandat2020botorch}
Maximilian Balandat, Brian Karrer, Daniel Jiang, Samuel Daulton, Ben Letham,
  Andrew~G Wilson, and Eytan Bakshy, `Botorch: a framework for efficient
  monte-carlo bayesian optimization', {\em Advances in neural information
  processing systems}, {\bf 33},  21524--21538, (2020).

\bibitem{butt2017globalized}
Khurram Butt, Ramhuzaini~A Rahman, Nariman Sepehri, and Shaahin Filizadeh,
  `Globalized and bounded nelder-mead algorithm with deterministic restarts for
  tuning controller parameters: Method and application', {\em Optimal Control
  Applications and Methods}, {\bf 38}(6),  1042--1055, (2017).

\bibitem{byrd2016stochastic}
Richard~H Byrd, Samantha~L Hansen, Jorge Nocedal, and Yoram Singer, `A
  stochastic quasi-newton method for large-scale optimization', {\em SIAM
  Journal on Optimization}, {\bf 26}(2),  1008--1031, (2016).

\bibitem{cheng2015distribution}
Dan Cheng and Armin Schwartzman, `Distribution of the height of local maxima of
  gaussian random fields', {\em Extremes}, {\bf 18}(2),  213--240, (2015).

\bibitem{dalibard2017boat}
Valentin Dalibard, Michael Schaarschmidt, and Eiko Yoneki, `Boat: Building
  auto-tuners with structured bayesian optimization', in {\em Proceedings of
  the 26th International Conference on World Wide Web}, pp. 479--488, (2017).

\bibitem{das2011real}
Swagatam Das, Sayan Maity, Bo-Yang Qu, and Ponnuthurai~Nagaratnam Suganthan,
  `Real-parameter evolutionary multimodal optimization—a survey of the
  state-of-the-art', {\em Swarm and Evolutionary Computation}, {\bf 1}(2),
  71--88, (2011).

\bibitem{de2000clonal}
L~Nunes De~Castro and Fernando~J Von~Zuben, `The clonal selection algorithm
  with engineering applications', in {\em Proceedings of GECCO}, volume 2000,
  pp. 36--39, (2000).

\bibitem{de2002artificial}
Leandro~Nunes De~Castro and Jon Timmis, `An artificial immune network for
  multimodal function optimization', in {\em Proceedings of the 2002 Congress
  on Evolutionary Computation. CEC'02 (Cat. No. 02TH8600)}, volume~1, pp.
  699--704. IEEE, (2002).

\bibitem{eriksson2019scalable}
David Eriksson, Michael Pearce, Jacob Gardner, Ryan~D Turner, and Matthias
  Poloczek, `Scalable global optimization via local bayesian optimization',
  {\em Advances in neural information processing systems}, {\bf 32}, (2019).

\bibitem{frazier2009knowledge}
Peter Frazier, Warren Powell, and Savas Dayanik, `The knowledge-gradient policy
  for correlated normal beliefs', {\em INFORMS journal on Computing}, {\bf
  21}(4),  599--613, (2009).

\bibitem{frazier2018tutorial}
Peter~I Frazier, `A tutorial on bayesian optimization', {\em arXiv preprint
  arXiv:1807.02811}, (2018).

\bibitem{griffiths2020constrained}
Ryan-Rhys Griffiths and Jos{\'e}~Miguel Hern{\'a}ndez-Lobato, `Constrained
  bayesian optimization for automatic chemical design using variational
  autoencoders', {\em Chemical science}, {\bf 11}(2),  577--586, (2020).

\bibitem{hernandez2014predictive}
Jos{\'e}~Miguel Hern{\'a}ndez-Lobato, Matthew~W Hoffman, and Zoubin Ghahramani,
  `Predictive entropy search for efficient global optimization of black-box
  functions', {\em Advances in neural information processing systems}, {\bf
  27}, (2014).

\bibitem{jameson1999re}
Antony Jameson, `Re-engineering the design process through computation', {\em
  Journal of Aircraft}, {\bf 36}(1),  36--50, (1999).

\bibitem{jiang2020efficient}
Shali Jiang, Daniel Jiang, Maximilian Balandat, Brian Karrer, Jacob Gardner,
  and Roman Garnett, `Efficient nonmyopic bayesian optimization via one-shot
  multi-step trees', {\em Advances in Neural Information Processing Systems},
  {\bf 33},  18039--18049, (2020).

\bibitem{krizhevsky2009learning}
Alex Krizhevsky, Geoffrey Hinton, et~al., `Learning multiple layers of features
  from tiny images'. Citeseer, (2009).

\bibitem{kushner1964new}
Harold~J Kushner, `A new method of locating the maximum point of an arbitrary
  multipeak curve in the presence of noise', (1964).

\bibitem{lai1985asymptotically}
Tze~Leung Lai, Herbert Robbins, et~al., `Asymptotically efficient adaptive
  allocation rules', {\em Advances in applied mathematics}, {\bf 6}(1),  4--22,
  (1985).

\bibitem{lee2018introduction}
John~M Lee, {\em Introduction to Riemannian manifolds}, volume 176, Springer,
  2018.

\bibitem{lewis2013nonsmooth}
Adrian~S Lewis and Michael~L Overton, `Nonsmooth optimization via quasi-newton
  methods', {\em Mathematical Programming}, {\bf 141}(1),  135--163, (2013).

\bibitem{li2015truss}
Jian-Ping Li, `Truss topology optimization using an improved species-conserving
  genetic algorithm', {\em Engineering Optimization}, {\bf 47}(1),  107--128,
  (2015).

\bibitem{liang2011genetic}
Yong Liang and Kwong-Sak Leung, `Genetic algorithm with adaptive
  elitist-population strategies for multimodal function optimization', {\em
  Applied Soft Computing}, {\bf 11}(2),  2017--2034, (2011).

\bibitem{lizotte2008practical}
Daniel~James Lizotte.
\newblock Practical bayesian optimization, 2008.

\bibitem{maclaurin2015gradient}
Dougal Maclaurin, David Duvenaud, and Ryan Adams, `Gradient-based
  hyperparameter optimization through reversible learning', in {\em
  International conference on machine learning}, pp. 2113--2122. PMLR, (2015).

\bibitem{movckus1975bayesian}
Jonas Mo{\v{c}}kus, `On bayesian methods for seeking the extremum', in {\em
  Optimization techniques IFIP technical conference}, pp. 400--404. Springer,
  (1975).

\bibitem{muller2021local}
Sarah M{\"u}ller, Alexander von Rohr, and Sebastian Trimpe, `Local policy
  search with bayesian optimization', {\em Advances in Neural Information
  Processing Systems}, {\bf 34},  20708--20720, (2021).

\bibitem{nguyen2022local}
Quan Nguyen, Kaiwen Wu, Jacob Gardner, and Roman Garnett, `Local bayesian
  optimization via maximizing probability of descent', {\em Advances in neural
  information processing systems}, {\bf 35},  13190--13202, (2022).

\bibitem{perrone2019constrained}
Valerio Perrone, Iaroslav Shcherbatyi, Rodolphe Jenatton, Cedric Archambeau,
  and Matthias Seeger, `Constrained bayesian optimization with max-value
  entropy search', {\em arXiv preprint arXiv:1910.07003}, (2019).

\bibitem{plessix2006review}
R-E Plessix, `A review of the adjoint-state method for computing the gradient
  of a functional with geophysical applications', {\em Geophysical Journal
  International}, {\bf 167}(2),  495--503, (2006).

\bibitem{rafiee2021intermittent}
Parisa Rafiee, Mustafa Oktay, and Omur Ozel, `Intermittent status updating with
  random update arrivals', in {\em 2021 IEEE International Symposium on
  Information Theory (ISIT)}, pp. 3121--3126. IEEE, (2021).

\bibitem{rasmussen2006gaussian}
Carl~Edward Rasmussen and Christopher K.~I. Williams, {\em Gaussian processes
  for machine learning}, Adaptive computation and machine learning, {MIT}
  Press, 2006.

\bibitem{shahriari2015taking}
Bobak Shahriari, Kevin Swersky, Ziyu Wang, Ryan~P Adams, and Nando De~Freitas,
  `Taking the human out of the loop: A review of bayesian optimization', {\em
  Proceedings of the IEEE}, {\bf 104}(1),  148--175, (2015).

\bibitem{snoek2012practical}
Jasper Snoek, Hugo Larochelle, and Ryan~P Adams, `Practical bayesian
  optimization of machine learning algorithms', {\em Advances in neural
  information processing systems}, {\bf 25}, (2012).

\bibitem{vargas2020bayesian}
Rodrigo~A Vargas~Hern{\'a}ndez, `Bayesian optimization for calibrating and
  selecting hybrid-density functional models', {\em The Journal of Physical
  chemistry. A}, {\bf 124}(20),  4053--4061, (2020).

\bibitem{wang2021network}
Su~Wang, Yichen Ruan, Yuwei Tu, Satyavrat Wagle, Christopher~G Brinton, and
  Carlee Joe-Wong, `Network-aware optimization of distributed learning for fog
  computing', {\em IEEE/ACM Transactions on Networking}, {\bf 29}(5),
  2019--2032, (2021).

\bibitem{wu2019hyperparameter}
Jia Wu, Xiu-Yun Chen, Hao Zhang, Li-Dong Xiong, Hang Lei, and Si-Hao Deng,
  `Hyperparameter optimization for machine learning models based on bayesian
  optimization', {\em Journal of Electronic Science and Technology}, {\bf
  17}(1),  26--40, (2019).

\bibitem{wu2019practical}
Jian Wu and Peter Frazier, `Practical two-step lookahead bayesian
  optimization', {\em Advances in neural information processing systems}, {\bf
  32}, (2019).

\bibitem{wu2017bayesian}
Jian Wu, Matthias Poloczek, Andrew~G Wilson, and Peter Frazier, `Bayesian
  optimization with gradients', {\em Advances in neural information processing
  systems}, {\bf 30}, (2017).

\bibitem{wu2023serverless}
Xidong Wu, Zhengmian Hu, Jian Pei, and Heng Huang, `Serverless federated auprc
  optimization for multi-party collaborative imbalanced data mining', in {\em
  SIGKDD Conference on Knowledge Discovery and Data Mining (KDD)}. ACM, (2023).

\bibitem{xie2017aggregated}
Saining Xie, Ross Girshick, Piotr Doll{\'a}r, Zhuowen Tu, and Kaiming He,
  `Aggregated residual transformations for deep neural networks', in {\em
  Proceedings of the IEEE conference on computer vision and pattern
  recognition}, pp. 1492--1500, (2017).

\bibitem{zhang2020bayesian}
Yichi Zhang, Daniel~W Apley, and Wei Chen, `Bayesian optimization for materials
  design with mixed quantitative and qualitative variables', {\em Scientific
  reports}, {\bf 10}(1),  1--13, (2020).

\bibitem{zhang2021two}
Yunxiang Zhang, Xiangyu Zhang, and Peter~I Frazier, `Two-step lookahead
  bayesian optimization with inequality constraints', {\em arXiv preprint
  arXiv:2112.02833}, (2021).

\bibitem{vzilinskas2006practical}
A~{\v{Z}}ilinskas.
\newblock Practical mathematical optimization: An introduction to basic
  optimization theory and classical and new gradient-based algorithms, 2006.

\end{thebibliography}
